\documentclass[a4paper,10pt]{article}

\usepackage{a4wide}
\usepackage{multicol, float}
\usepackage{lineno}
\usepackage{verbatim, euscript}
\usepackage{subfigure}
\usepackage{bbm}

\usepackage{pictexwd}

\usepackage{psfrag}

\usepackage{tikz}
\usetikzlibrary{arrows.meta}
\usepackage{rotating}
\usepackage{color}
\usepackage{mathrsfs}
\usepackage{pictexwd}
\usepackage{tikz}
\usepackage{etex} 
\usepackage{multido}

\usepackage{lineno,hyperref}
\usepackage{verbatim, euscript}
\usepackage{subfigure}
\usepackage{bbm,amssymb}
\usepackage{amsfonts}
\usepackage{amsmath, mathtools}
\usepackage{pictexwd}
\usepackage{amsthm}
\usepackage{graphicx}

\usepackage{tikz}
\usetikzlibrary{arrows.meta}
\usetikzlibrary{math}
\usetikzlibrary{decorations.pathreplacing}

\usepackage{rotating}
\usepackage{color}
\usepackage{mathrsfs}

\usepackage{mathtools}
\mathtoolsset{showonlyrefs}

\usetikzlibrary{arrows,automata}

\numberwithin{equation}{section}

\newcommand{\eps}{\varepsilon}

\theoremstyle{plain}
\newtheorem{theorem}{Theorem}[section]
\newtheorem{lemma}[theorem]{Lemma}
  
\newtheorem{corollary}[theorem]{Corollary}
\newtheorem{definition}[theorem]{Definition}
\newtheorem{remark}[theorem]{Remark}

\newtheorem{proposition}[theorem]{Proposition}

\newcommand{\E}{\mathbb{E}}
\newcommand{\PP}{\mathbb{P}}

\newcommand{\1}{\mathbbm{1}} 
\newcommand{\N}{\mathbb{N}}
\newcommand{\Z}{\mathbb{Z}}
\newcommand{\R}{\mathbb{R}}

\newcommand{\EE}[1]{\mathbb{E} \left[ #1 \right]}

\newcommand{\pp}[1]{\mathbb{P} \left( #1 \right)}
\newcommand{\blue}[1]{\textcolor{black}{#1}}

\begin{document}
\title{Haldane's formula in Cannings models: \\ The case of  moderately weak selection}
	\author{Florin Boenkost, Adri\'an Gonz\'{a}lez Casanova, \\ Cornelia Pokalyuk,  Anton Wakolbinger}
	\maketitle

\abstract{We introduce a Cannings model with directional selection via a paintbox construction and establish a strong duality with the line counting process of a new \emph{Cannings ancestral selection graph} in discrete time. This duality also yields  
		a formula for the fixation probability of the beneficial type. Haldane's formula states that for a single selectively advantageous individual in a population of haploid individuals of size $N$ the prob\-ability of fixation is asymptotically (as $N\to \infty$) equal to the selective advantage of haploids $s_N$  divided by half of the offspring variance. For a class of offspring distributions within Kingman attraction we prove this asymptotics for sequences $s_N$ obeying $N^{-1} \ll s_N \ll N^{-1/2} $, which is a regime of ``moderately weak selection''. It turns out that  for $ s_N \ll N^{-2/3} $ the Cannings ancestral selection graph is so close to  the ancestral selection graph of a Moran model that a suitable coupling argument allows  to play the problem back asymptotically to the fixation probability in the Moran model, which can be computed explicitly.}

\section{Introduction}

	In population genetics the standard model for the neutral reproduction of  haploid individuals in  discrete time  and with constant population size $N$  is the classical \emph{Wright-Fisher model}: the offspring numbers from one generation to the next arise by throwing $N$ times a die with $N$ faces and are thus   Multinomial$(N;1/N,\ldots, 1/N)$-distributed. This is a special instance within the 
general class of \emph{Cannings models}, see \cite{Cannings1974} and \cite[Chapter 3.3]{Ewens2004}, where the offspring numbers are assumed to be exchangeable and to sum to $N$. 
	
	In the Wright-Fisher model \emph{directional selection} can be included by appropriately biasing the weights of individuals that have a selectively beneficial type or a wildtype. 
	While for general Cannings models it is not completely clear how to incorporate selection, a biasing of weights can be done in a natural way for the large class of Cannings models that admit a \emph{paintbox representation} (in the sense that the $N$-tuple of offspring numbers is \emph{mixed multinomial}, see Section \ref{paintneutral}).
	
	 For this class of models we introduce a graphical representation which extends to the case with directional selection, and leads to a time discrete version of the ancestral selection graph that was developed by Krone and Neuhauser in \cite{KN} for the (continuous time) Moran model. 
	 
Recently  Gonz\'{a}lez Casanova and Span{\`o} constructed in \cite{GS} an ancestral selection graph for a special class of Cannings models. While their construction relies on analytic arguments,  we provide here a probabilistic construction which works for a wider class of models and also gives a clear interpretation of the  role of the geometric distribution of the number of potential parents in this context. This construction will be explained in Section \ref{DASG}. We will prove a sampling duality between the Cannings frequency process and the line counting process of the discrete ASG (alias \emph{Cannings ancestral selection process} or \emph{CASP}), see Theorem \ref{Theorem Duality}. This also allows to obtain a handsome representation of the fixation probability of the beneficial type in terms of the expected value of the CASP in equilibrium, see Corollary \ref{corfix}.

	The calculation of fixation probabilities is a prominent task in mathematical population genetics; for a review including a historical overview  see \cite {PW}.  A classical idea going back to Haldane, Fisher and Wright (see \cite{Haldane1927}, \cite{Fisher1922} and \cite{Wright1931}) and known as \emph{Haldane's formula},  is to approximate the probability of fixation of a beneficial allele with small selective advantage $s$  by the survival probability $\pi(s)$ of a supercritical, near-critical Galton-Watson branching process, 
	\begin{equation}\label{Haldane0}
	\pi(s) \sim \frac {s}{\rho^2/2} \,\, \mbox{ as } s\to 0
	\end{equation}
	where $\rho^2$ is the offspring variance and $1+s$ is the expected offspring size.
	
 In Remark~\ref{remark36}b) we will briefly discuss perspectives and frontiers of a derivation of \eqref{Haldane0} in terms of a branching process approximation.
    Couplings with Galton-Watson processes were used by Gonz\'{a}lez Casanova et al.\ in \cite{GKWY} to prove that  \eqref{Haldane0} indeed gives the asymptotics for the fixation probability for a class of Cannings models (with mixed multi-hypergeometric offspring numbers) that arise in the context of experimental evolution. 	This was achieved under the assumption that $s_N\sim N^{-b}$ with $0<b<1/2$, i.e. for a moderately strong selection. There, the question remained open if \eqref{Haldane0} also captures the asymptotics of the fixation probability for  $s_N\sim N^{-b}$ with $1/2\le b <1$. For the case $1/2 <  b <1$, Theorem \ref{Haldane  Formel} gives an affirmative answer for subclasses of Cannings models admitting a paintbox representation, and in particular also for the Wright-Fisher model with selection. In Theorem \ref{Haldane  Formel}a) we prove \emph{Haldane's formula} under the condition
\begin{equation}\label{sasymptotics}
	N^{-1+\eta}\le s_N \le   N^{-2/3-\eta}  
	\end{equation} 
	(with $\eta>0$)
	for a class of paintboxes that satisfy in particular M\"ohle's condition (which guarantees that the coalescents of the neutral Cannings model are in the domain of attraction of Kingman's coalescent, see \cite{M1}).
Under these  assumptions we show in Section~\ref{Section DASP} that  the CASP is close to the ASG line counting process of a corresponding Moran model over a long period of time. Indeed, for a Moran model with directional selection, a  representation of the fixation probability in terms of the  ASG line counting process is valid, and the fixation probability can be calculated explicitly; this we explain in Section \ref{Section Moran Model}. 
Relaxing \eqref{sasymptotics}, in Theorem \ref{Haldane  Formel} b) we prove \emph{Haldane's formula} for sequences $(s_N)$ with  
	\begin{equation}\label{sasymptoticsa}
	N^{-1+\eta}\le s_N \le   N^{-1/2-\eta}  
	\end{equation}
under  more restrictive moment conditions on the paintbox. 	
	Examples fulfilling these moment conditions include the Wright-Fisher case as well as paintboxes that are of Dirichlet-type, for more details see Section \ref{HaldaneCannings}. The main tool of the proof under these conditions is a concentration result on the equilibrium distribution of the CASP, see Section 7. This yields a sufficiently good estimate of the expected value of the CASP in equilibrium to show Haldane's formula by means of the above-mentioned Corollary \ref{corfix}.

		\section{Cannings models and Moran models with selection}
	\subsection{A paintbox representation for the neutral reproduction}\label{paintneutral}
	In a neutral Cannings model with population size $N$, the central concept is the exchangeable \mbox{$N$-tuple} $\nu =(\nu_1, \ldots, \nu_N)$ of \emph{offspring sizes}, with non-negative integer-valued components summing to $N$. A reasonably large class of such random variables $\nu$  admits a \emph{paintbox construction}, i.e.\ has a mixed multinomial distribution with parameters $N$ and $\mathscr{W}$, where 
	$\mathscr{W}=(W_1,W_2,\dots ,W_N)$ is an exchangeable random $N$-tuple of probability weights taking its values in  \[\Delta_N=\left\lbrace  (x_1,x_2,\dots x_N) :  x_i \ge 0, \sum_{i=1}^N x_i =1 \right\rbrace.\] 
	While this is clearly reminiscent of Kingman's paintbox representation of  exchangeable partitions of $\mathbb N$, here we are dealing with a finite $N$. As such,
	obviously, not all exchangeable offspring sizes are mixed multinomial -- consider  e.g.\ a uniform permutation of the vector $(2,\dots,2,0,\dots,0)$. On the other hand, the exchangeable mixed multinomials cover a wide range of applications; e.g., they can be seen as approximations of the offspring sizes in a model  of experimental evolution, where at the end of each reproduction cycle $N$ individuals are sampled \emph{without replacement} from a union of $N$ families with large i.i.d.\ sizes; see \cite{GKWY} and \cite{BGPW},  where the distribution of the family sizes was assumed to be geometric with expectation $\gamma =100$. This leads to a mixed multi-hypergeometric offspring distribution, whose analogue for $\gamma = \infty$ would be a mixed multinomial offspring distribution with $\mathscr L(\mathscr W)$ the Dirichlet($1,\ldots,1$)-distribution on~$\Delta_N$.	
	Let us now briefly review the graph of genealogical relationships in a Cannings model.  In each generation $g$, the individuals are numbered by $i \in [N]:= \left\lbrace1,\dots,N \right\rbrace$ and denoted by $(i,g)$.  A \emph{parental relation} between individuals in generation $g$ and $g-1$ is defined in the following way. Let $\mathscr{W}^{(g)}, g\in \mathbb Z$, be i.i.d.\ copies of $\mathscr W$.   Every individual $(j,g)$ is assigned a parent $(V_{(j,g)},g-1)$ in generation $g-1$ by means of an $[N]$-valued  random variable $V_{(j,g)}$ with conditional distribution $\PP(V_{(j,g)}=i|\mathscr{W}^{(g-1)})=W_i^{(g-1)}$, $i \in [N]$. For each $g\in \mathbb Z$, the random variables $V_{(j,g)}$, $j=1,\ldots, N$, are assumed to be independent given $\mathscr{W}^{(g-1)}$.  Also, for each $g \in \mathbb Z$, due to the exchangeability of $(W_1^{(g-1)},\dots,W_N^{(g-1)} )$, the random variables $V_{(1,g)},\dots ,V_{(N,g)}$ are uniformly distributed on $[N]$, and in general are correlated.  With this construction within one generation step we produce an exchangeable $N$-tuple of \emph{offspring sizes}, i.e.~the number of children for each individual $(i,g-1)$, $i \in [N]$. Due to the assumed independence of the random variables $\mathscr{W}^{(g)}, g\in \mathbb Z$, the offspring sizes as well as the ``assignments to parents'' $(V_{(1,g)},\dots ,V_{(N,g)})$ are independent along the generations $g$.
	
	\blue{Later in the paper we will deal with a sequence of Cannings models indexed by $N$, which will come along with a sequence $(\mathscr W^{(N)})$ and i.i.d. copies $\mathscr W^{(N,g)}$ of $\mathscr W^{(N)}$. For notational simplicity  we will sometimes suppress the superscript $N$ and keep writing $\mathscr W$ and $\mathscr W^{(g)}$ instead of $\mathscr W^{(N)}$ and $\mathscr W^{(N,g)}$; this will not lead to any ambiguities.}

	\subsection{A paintbox representation incorporating selection}\label{paintsel}	
	We now build directional selection with strength $s_N \in (0,1)$ into the model. Assume that  each individual has one of two types, either the \emph{beneficial type} or the \emph{wildtype}. 
	Let the chances to be chosen as a parent be modified by decreasing the weight of each wildtype individual by the factor $1-s_N$. 
	In other words, if individual $(i,g)$ has the wildtype the weight reduces to $\widetilde{W}_i:=(1-s_N) W_i$ and if the individual has the beneficial type the weight remains $\widetilde{W}_i :=W_i$.  Let $\widetilde {\mathscr W}^{(g)}:= (\widetilde{W}^{(g)}_1, \ldots,   \widetilde{W}^{(g)}_N)$. Given the type configuration in generation $g-1$, the parental relations are now generated in a two-step manner: First, assign the random weights $\widetilde{\mathscr W}^{(g-1)}$ to the individuals in generation $g-1$, then follow the rule
	\begin{equation} \label{model}
	\mathbb P((i,g-1) \mbox{ is parent of } (j,g)\, \mid \widetilde{\mathscr W}^{(g-1)}) = \frac{\widetilde{W}^{(g-1)}_i}{\sum_{\ell=1}^{N} \widetilde{W}^{(g-1)}_\ell}\,.
	\end{equation}
	Individual $(j,g)$ then inherits the  type from its parent. Note that $\widetilde{\mathscr W}^{(g-1)}$ is measurable with respect to  $\mathscr W^{(g-1)}$ together with the type configuration in generation $g-1$. Because of the assumed exchangeability of the $W^{(g-1)}_i$, $i=1,\ldots, N$, the distribution of the type configuration in generation~$g$ only depends on the \emph{number} of  individuals in generation $g-1$ that carry the beneficial type.
	Thus, formula \eqref{model} defines a Markovian dynamics for the type frequencies. We will denote the number of wildtype individuals in generation $g$ by $K_g$, and will call $(K_g)_{g=0,1,\ldots}$ a \emph{ Cannings frequency process} with parameters $N$, $\mathscr L(\mathscr W)$ and $s_N$. In particular, \eqref{model} implies that  given $\{K_{g-1}=k\}$, $K_g$ is  mixed Binomial with parameters $N$ and $P(k,\mathscr W)$, where
\begin{equation}\label{wildtypeprob}	
	P(k,\mathscr W)= \frac{(1-s_N)\sum_{i=1}^k W_i}{(1-s_N)\sum_{\ell=1}^k W_\ell +  \sum_{\ell=k+1}^N W_\ell}.
	\end{equation}
	\subsection{The Cannings ancestral selection process}\label{secCASP}
	Again let $N \in \mathbb N$, $\mathscr W$ as in Section \ref{paintneutral}, and $s_N \in (0,1)$. The \emph{Cannings ancestral selection process (CASP)} $\mathcal{A}=(A_m)_{m=0,1,\ldots}$ with parameters $N$, $\mathscr L(\mathscr W)$ and $s_N$ counts the number of potential ancestors in generation $g-m$ of a sample taken in generation $g$. We will give a graphical representation in Section \ref{DASG}; in the present section we {\em define} the CASP as an $[N]$-valued Markov chain whose one-step transition is composed of a branching and a coalescence step as follows: 
	Given $A_m=a$, the branching step takes $a$ into a sum $H=\sum_{\ell=1}^a G^{(\ell)}$ of independent Geom($1-s_N$)-random variables; in other words, the random variable $H$ has a negative binomial distribution with parameters $a$ and $1-s_N$, and thus takes its values in $\{a,a+1,\ldots\}$. (Here and below, we understand a Geom($p$)-distributed random variable as describing the number of trials (and not only failures) up to and including the first success in a coin tossing with success probability $p$.)
	The coalescence step arises (in distribution) through a two-stage experiment: first choose a random $\mathscr W$ according to the prescribed distribution $\mathscr L(\mathscr W)$,  then, given $\mathscr W$ and the number $H$ from the branching step,  place $H$ balls independently into $N$ boxes, where $W_i$ is the  probability that the first (second,\ldots, $H$-th) ball is placed into the $i$-th box, $i=1, \ldots, N$. The random variable $A_{m+1}$ is distributed as the number of non-empty boxes.\\
	\blue{To emphasize the dependence of $\mathcal{A}$ on $N$ we sometimes will write $\mathcal{A}^{(N)}=(A_m^{(N)})_{m \geq 0}$ in place of $\mathcal A$.}

		\subsection{\blue{Haldane's formula in the Moran model}} \label{Section Moran Model}
	In a two-type Moran model with constant population size $N$ and directional selection (see e.g.\ \cite[Chapter 6]{RD}), 
	each individual reproduces \blue{(in continuous time)} at a constant rate $\gamma/2, \gamma > 0$, and individuals of the beneficial type reproduce with an additional rate $s_N>0$.
	Let $Y_t^{(N)}$ be the number of wildtype individuals at time $t$, and let
	$(B_r^{(N)})_{r \geq 0}$ be the counting process of potential ancestors traced back from some fixed time. The process $(B^{(N)}_r)_{r \geq 0}$, which we call the Moran ancestral selection process (or MASP for short), is a Markov jump process with jumps from
	$k$ to $k+1$ at rate $k s_N \frac{N-k}{N}$ for $1\leq k \leq N-1$  and from 
	$k$ to $k-1$ at rate $\frac{\gamma}{N} \binom k 2$, see \cite{KN}. The well-known graphical representation of the Moran model yields a strong duality between $(Y_t^{(N)})_{t\geq 0}$  and $(B_r^{(N)})_{r\geq 0}$. Stated in words this says that a sample $J\subset [N]$ at time $t$ consists solely of wildtype individuals if and only if all the potential ancestors of $J$ are wildtype. This immediately leads to the following (hypergeometric) sampling duality	
	\begin{align}\label{Moran Duality}
	& \hspace{-2.8cm}\E\left[\frac{Y_t^{(N)}(Y_t^{(N)}-1)  \cdots(Y_t^{(N)}-(n-1))}{N(N-1)\cdots(N-(n-1))}\Big | Y_0^{(N)}=k\right] \\ &\quad \quad =\mathbb E\left[\frac{k(k-1)\cdots(k-(B^{(N)}_t-1))}{N(N-1)\cdots(N-(B_t^{(N)}-1))}\Big | B_0^{(N)}=n\right]
	\end{align}
	where $t\geq 0 $ and $k,\, n \in [N]$.
	Specializing \eqref{Moran Duality} to $k=N-1$ and $n=N$ gives
	\begin{equation}\label{specialMoran}
	\mathbb P(Y^{(N)}_t=0  | Y^{(N)}_0=N-1) = 1-\mathbb E\left[\frac{B^{(N)}_t}{N} \Big | B^{(N)}_0=N\right].
	\end{equation}
	Taking the limit $t\to \infty$ in \eqref{specialMoran} leads to
	\begin{equation}\label{fixNMoran}
		\pi^M_N:= \lim_{g\to \infty} \mathbb P[Y^{(N)}_t=0  | Y^{(N)}_0=N-1] = \mathbb E\left[\frac{B^{(N)}_{\rm eq}}{N}\right],
		\end{equation}
	where $B^{(N)}_{\rm eq}$ denotes the stationary distribution of $(B^{(N)}_r)_{r\geq 0}$.
	As observed in \cite{C},   
	$B_{\rm eq}^{(N)}$  
	is a binomially distributed random variable  with parameters $N$ and 
	$p_N:= \frac{2s_N}{2s_N +\gamma}$ that is conditioned not to vanish.
	\\
	In particular,
	\begin{align}\label{HaldaneMoran}
	\frac{\EE{B_{\rm eq}^{(N)}}}{N} =p_N \frac{1}{1-(1-p_N)^N}.\end{align}
	For $s_N=\frac{\alpha}{N}$, $\alpha >0$ (the case of weak selection) this specializes to Kimura's formula \cite{K}
	\begin{equation*}
	\pi_N^M=  \frac{2 s_N}{\gamma} \frac{1}{1-e^{-\frac{2 \alpha}{\gamma} }} (1+o(N^{-1})).
	\end{equation*}
	For $N^{-\eta} \ge s_N \ge N^{-1+\eta}$ (the case of moderate selection) we obtain Haldane's formula 
	\begin{align*}
	\pi_N^M= \frac{2 s_N}{\gamma} (1+o(s_N))
	\end{align*}
	and for $s>0$ (the case of strong selection) this results in
	\begin{align*}
	\pi_N^M= \frac{2 s}{2s+\gamma} (1+O(N^{-1})).
	\end{align*}

	\section{Main results} \label{Main Results}
	\subsection{Duality of Cannings frequency and ancestral selection process}\label{DCA}
 For $N \in \mathbb N$, $\mathscr W= \mathscr W^{(N)}$ as in Section \ref{paintneutral}, and $s_N \in (0,1)$, let $(K_g)_{g\geq 0}=((K^{(N)}_g)_{g\geq 0})$ be the Cannings frequency process with parameters $N$,  $\mathscr L(\mathscr W)$ and $s_N$ as defined in Section \ref{paintsel}, and let $(A_m)_{m\geq 0}$ be the Cannings ancestral selection process with parameters $N$,  $\mathscr L(\mathscr W)$ and $s_N$ as defined in Section~\ref{secCASP}.
	\begin{theorem} [Sampling duality] \label{Theorem Duality} Let $g \geq 0$, and $k, n \in [N]$. Let $\mathcal J_n$ be uniformly chosen from all subsets of $[N]$ of size $n$, and given $A_g=a$, $a\in [N]$,  let $\mathscr A_g$ be uniformly chosen from all subsets of  $[N]$ of size $a$. Then we have the following duality relation
		\begin{align}
		\mathbb P(\mathcal J_n \subset [K_g] \mid K_0=k) = \mathbb P(\mathscr  A_g \subset [k] \mid A_0 = n). \label{Duality with J}
		\end{align}
	\end{theorem}
	\begin{remark}\label{strongdual}
		 A strong (pathwise) version of the duality relation~\eqref{Duality with J} will be provided by formula \eqref{formal2} in Section \ref{DASG}, which roughly spoken says that ``A sample from generation $g$ is entirely of wildtype if and only if all of its potential ancestors in generation $0$ are of wildtype''.
	\end{remark}
	Expressed in terms of $K_g$ and  $A_g$,  the sampling duality relation \eqref{Duality with J}  reads \blue{(as in the case of the Moran model, see Section \ref{Section Moran Model})}
	\begin{align}
	\mathbb E \left[\frac{K_g(K_g-1)\cdots  (K_g-n+1)}{N(N-1)\cdots(N-n+1)}\big| K_0=k\right] = \mathbb E\left[\frac{k(k-1)\cdots  (k-A_g+1)}{N(N-1)\cdots(N-A_g+1)}\big| A_0=n\right]. \label{CASP Hypergeometric Duality}
	\end{align}
	\blue{In the very same way as we derived \eqref{fixNMoran} from \eqref{Moran Duality} we obtain (first 
	specializing \eqref{CASP Hypergeometric Duality} to $k=N-1$ and $n=N$ and then
	taking the limit $g\to \infty$) the following}
	\begin{corollary}\label{corfix}
		Let $A_{\rm eq}$ have the stationary distribution of the Cannings ancestral selection process $(A_m)_{m\geq 0}$. The fixation probability of a single beneficial mutant is
		\begin{equation}\label{fixN}
		\pi_N:= \lim_{g\to \infty} \mathbb P[K_g=0  | K_0=N-1] = \mathbb E\left[\frac{A_{\rm eq}}{N}\right].
		\end{equation}
	\end{corollary}
	\begin{remark}
		In the light of Remark \ref{strongdual}, the representation \eqref{fixN} can be interpreted as follows: With a single beneficial mutant in generation $0$, the beneficial type goes to fixation if and only if the beneficial mutant  is among the potential ancestors in generation $0$ of the population at a late generation $g$. In the limit $g\to \infty$ the number of these potential ancestors is distributed as $A^{(N)}_{\rm eq}$, and given  $A^{(N)}_{\rm eq}$, the probability that the beneficial mutant is among them is $\frac{A^{(N)}_{\rm eq}}{N}$.
	\end{remark}
	\subsection{Haldane's formula for Cannings models with selection}\label{HaldaneCannings}
	Let $(s_N)$ be a sequence in $(0,1)$ that satisfies \eqref{sasymptotics}.  For each $N$ let $\mathscr W^{(N)}=(W_1^{(N)}, \ldots, W_N^{(N)}) $ be as in Section \ref{paintneutral}, and assume that for some $\rho^2\geq 1$ 
\begin{equation}\label{second}
	\EE{\left(W_1^{(N)}\right)^2}=\frac{\rho^2}{N^2}+O(N^{-3}),  	
	\end{equation}	
 with the $O(\cdot)$-terms  referring to $N\to \infty$. (The requirement $\rho^2\geq 1$ is natural because $\EE{W_1^{(N)}}=\frac 1N$.) We consider the following condition:
\begin{equation}\label{third}
	\EE{\left(W_1^{(N)}\right)^3}=O(N^{-3})
	\end{equation}	
 Note that \eqref{second} and \eqref{third} imply
	\begin{equation} \label{Moehle} N \EE{\left(W_1^{(N)}\right)^2} \to 0 \quad  \mbox{and}  \quad
	\EE{\left(W_1^{(N)}\right)^3} =o\left(\EE{\left(W_1^{(N)}\right)^2}\right).
	\end{equation}
For Cannings processes admitting a paintbox representation, {\color{black} \eqref{Moehle} is equivalent to M\"ohle's condition}, see \cite{M1}, which, in turn, is equivalent to the neutral Cannings coalescent being in the domain of attraction of a Kingman coalescent as $N\to \infty$.
We also consider the following condition which (see Remark \ref{remark36} a))
 is not implied by \eqref{second} and \eqref{third}:\\
There exists a  sequence $(h_N)$ of natural numbers with the properties
\begin{equation} \label{ellN}
h_N\to \infty \quad \mbox{and} \quad h_N=o(\ln N) \quad \mbox{as  } N\to \infty
\end{equation}
such that 
for all sufficiently large $N$ and all $n \le 2 h_N$
{\color{black}
\begin{equation} \label{moments}
\EE{\left(W_1^{(N)}\right)^{n}} \le \left(\frac{Kh_N}{N}\right)^n.
\end{equation}}
for some $K \geq 1$ not depending on $N$ and $n$.
	\begin{theorem} [Haldane's formula] \label{Haldane  Formel}\qquad \\
Consider a sequence of Cannings frequency processes $(K_g^{(N)} )_{g \geq 0}, \, N \geq 1$, with parameters $\mathscr L(\mathscr W^{(N)})$ and~$s_N$. Assume that the selection is moderately weak in the sense that $(s_N)$ satisfies \eqref{sasymptoticsa}, and assume that $\mathscr L(\mathscr W^{(N)})$ satisfies \eqref{second} with $\rho \ge 1$ and \eqref{third}. Then the fixation probabilities $\pi_N$ of single beneficial mutants follow the asymptotics
		\begin{equation}
		\pi_N= \frac{2 s_N}{\rho^2 } + o(s_N) \quad \mbox{as }N\to \infty \label{Haldane}
		\end{equation}
provided one of the following additional requirements a) or b) is satisfied:
		\begin{itemize}
		\item[a)] $(s_N)$ satisfies \eqref{sasymptotics}, 
		\item[b)] \blue{the moments of $(W_1^{(N)})$ obey \eqref{moments}, with a sequence $(h_N)$ that satisfies \eqref{ellN}.}
			\end{itemize}
	\end{theorem}		

	The proof of part a) will be given in Section \ref{Section DASP} and that of part b) in Section \ref{Sec7}. Here we give brief sketches of the proofs.
	{\color{black} Our method of proof of Theorem 	\ref{Haldane  Formel} employs Corollary \ref{corfix}, for part a) by a coupling of the Cannings ancestral selection process (CASP) $(A_m)$ and the corresponding ancestral selection process of the Moran model (MASP) $(B_r)$, and for part b) via  a concentration analysis of the equilibrium distribution of $(A_m)$. In both parts it is essential that $s_N$ is not too large; this is guaranteed by Conditions  \eqref{sasymptotics} and \eqref{sasymptoticsa}, respectively.
		As described in Section \ref{secCASP}, the size of $s_N$ is responsible for the size of the upward jumps of $(A_m)$, and thus        it turns out that the random variable $A_{\rm eq}^{(N)}$ depends on $s_N$ in a stochastically monotone way.

In particular, our analysis of $A_{\rm eq}^{(N)}$ which makes Corollary 3.3 applicable for the derivation of (3.10) relies on the assumption $s_N \ll N^{-1/2}$.
	 
	 Part a): Coupling of CASP and MASP\\
	  We know from Section \ref{Section Moran Model} that the asymptotics \eqref{Haldane} holds for the fixation probabilities $\pi_N^M$ (starting from a single beneficial mutant) in  a sequence of Moran$(N)$-models with neutral reproduction rate $\rho^2/2$ (or equivalently with pair coalescence probability $\rho^2$) and selection strength~$s_N$. 
	We show in Section \ref{Section DASP} that thanks to Conditions \eqref{second} and \eqref{third} and Assumption \eqref{sasymptotics}, we can couple $(A_m^{(N)})_{m\ge 0}$ and $(B_r^{(N)})_{r\ge 0}$  long enough to ensure $\EE{A_{\rm eq}^{(N)}}=\EE{B_{\rm eq}^{(N)}}(1+o(1))$, 
	  which proves Theorem \ref{Haldane  Formel}a). 
	 
			The relevance of Condition \eqref{sasymptotics} for our proof of part a) of Theorem 	\ref{Haldane  Formel} can heuristically be seen as follows. An inspection of the jump probabilities described in Section \ref{secCASP} shows that in a regime of negligible multiple collisions the quantity $a^{(N)}_{\rm eq} := N \frac{s_N}{\rho^2/2}$ is an asymptotic center of attraction for the dynamics of $A^{(N)}$.  Since $B^{(N)}$ is close to its equilibrium after a time interval of length $s_N^{-(1+\delta)}$ we require the coupling  of $A^{(N)}$ and $B^{(N)}$  to hold for $s_N^{-(1+\delta)}$ many generations. This works if within this time interval the number of potential ancestors $A_m^{(N)}$ makes at most jumps of size $1$ in each generation. 
 The number of potential parents decreases by $1$ if a single pair coalesces. Near  $a^{(N)}_{\rm eq} $ the number of pairs within  $A_{\rm eq}^{(N)}$ is of the order $(a^{(N)}_{\rm eq})^2$. Hence, the probability of a pair coalescence per generation is of the order $(a^{(N)}_{\rm eq})^2/N$, and the probability of more than a single pair coalescence per generation is of the order $(a^{(N)}_{\rm eq})^4/N^2$, or equivalently, of the order $N^2 s_N^4$, since this probability can be estimated by the probability for two pair coalescences, see Lemma \ref{Lemma jumpsize=1}. Analogously, the probability that the number of potential parents increases by more than $1$ is also of order $N^2s_N^4$. Hence, the probability that jumps at most of size 1 occur for $s_N^{-(1+\delta)}$ generations, that have to be considered, is of the order $(1- N^2 s_N^4)^{s_N^{-(1+\delta)}}$. This probability is high, if $N^2 s_N^{3-\delta}\ll 1$, which corresponds to the upper bound in \eqref{sasymptotics}.

	Part b): 
	Concentration analysis of the CASP equilibrium distribution\\ 
		To show that the expectation of the CASP in equilibrium is $\frac{2}{\rho^2} N s_N (1 +o(1))$, we show in Lemmata \ref{Lemma Coming down} and \ref{Lemma going up from 1} that the CASP needs an at most polynomially long time{\color{black}, i.e. a time of order $N^c$ for some $c >0$,} to enter a central region (i.e. an interval of moderately large size around the center of attraction) of the CASP from outside, but (as proved in Proposition \ref{Lemma exp leaving time}) does not leave this central region up to any polynomially long time with sufficiently high probability. For this purpose we couple $\mathcal{A}$ with a random walk that makes jumps only of limited size. To show that large jumps (upwards or downwards) are negligible, we make use of the assumption $s_N \ll N^{-1/2}$. The probability that the CASP in a state not too far from the ``center'' (that is in a state of the order of $N s_N$) makes a jump at least of size $h_N$ (with $1\ll  h_N \ll N s_N$) upwards can essentially be estimated by the probability 
		that at least $h_N$ individuals each give rise to at least two branches in the branching step (described in Section \ref{secCASP}). The probability that in the branching step an individual gives rise to at least two branches is $\approx s_N$. Thus the probability that for individuals $1, ..., h_N$ at least two  branches are generated at the branching step is $\approx s_N^{h_N}$.
		There are  $(N s_N)! /( N s_N -h_N)! = \Theta( (N s_N)^{h_N})$ possibilities to choose $h_N$ individuals out of $s_N N$ many individuals. Consequently, the probability for a jump of at least of size $h_N$ upwards can be estimated from above by $O(s_N^{h_N} (N s_N)^{h_N}) = O((N s_N^2) ^{h_N})$, which is of negligible size if $s_N \ll N^{-1/2}$. 
	}
	
	   {\color{black}To arrive at a similar lower bound for the probability of large jumps downwards Condition \eqref{moments} is applied. Higher moments of the weight $W_1$ have to be controlled in this case, since for a downwards jump several individuals have to choose the same parent(s), for more details see Lemma \ref{Lemma Log N Down}.} 
	
	\begin{remark}\label{remark36}
		\begin{itemize} 
		\item[a)] {\color{black}(Conditions on the weights)
			 To illustrate a hierarchy between the conditions on the weights, consider the class of examples indexed by $n=1,2,...$, where in each generation with probability $N^{-n+\frac{1}{2}}$ an uniformly chosen individuals gets weight $1$ and with probability $1-N^{-n+\frac{1}{2}}$ every individual has weight $\frac{1}{N}$. For $n=1$ neither Condition \eqref{second} nor Condition \eqref{Moehle} is  fulfilled, for $n= 2$ Conditions \eqref{second} and  \eqref{Moehle} are fulfilled but  Condition \eqref{third} is violated,  and for $n=3$ Conditions \eqref{second} and \eqref{third} are fulfilled but Condition $\eqref{moments}$ fails.\\
			 For examples of weights $(\mathscr W^{(N)})$ that fulfil the requirements of Theorem \ref{Haldane  Formel}b), see Lemma \ref{Lemma implying Conditions}.}			
	\item[b)] (Moderately strong selection and Galton-Watson approximations) A regime for which it is possible to derive \eqref{Haldane} by means of a Galton-Watson approximation is that of {\em moderately strong selection} $1\gg s_N \gg N^{-1/2}$. A proof of this assertion  (under  somewhat more restrictive moment conditions on the weights $\mathscr W$ than those in Theorem \ref{Haldane  Formel} b) ) is the subject of the paper \cite{BoGoPoWa2}; see also the discussion in the paragraph following \eqref{Haldane0} in the Introduction. Together with the approach of the present paper this  does not yet cover the case $s_N\sim  N^{-1/2}$; we conjecture that Haldane's formula is valid also for this particular exponent.
				
			Here is a quick argument which explains the relevance of the exponent $1/2$ as a border for the applicability of a Galton-Watson approximation. The beneficial type is with high probability  saved from extinction if the number of individuals of the beneficial type exceeds (of the order of) $s_N^{-1}$. Hence, for a proof via approximations with Galton-Watson processes one wants couplings of the CASP with GW-processes to hold until this number of beneficial individuals is reached. However, a Galton-Watson approximation works only until there is an appreciable amount of ``collisions'' between the offspring of the beneficial individuals in a branching step, since collisions destroy independence. By well-known ``birthday problem'' considerations, such an amount of collisions happens as soon as there are (of the order of) $N^{1/2}$ beneficial individuals in the population. Consequently, for the GW-approach we require $s_N^{-1} \ll N^{1/2}$.           
 \item[c)] (Possible generalisations) The introduced duality method for Cannings models with selection may well prove beneficial also in more general settings. The construction of the Cannings ancestral selection graph given in Section \ref{DASG} can for example also be carried out in a many-island situation,
 with migration between islands in discrete or continuous time. This should then lead to  generalizations of Theorems \ref{Theorem Duality} and \ref{Haldane  Formel}. Under Assumption \eqref{sasymptotics}, and with an appropriate scaling of the migration probabilities, one might expect that the Cannings ancestral selection graph is again close to the (now structured) Moran ancestral selection graph.
 \end{itemize}
	\end{remark}
	\blue{ 
	We conclude this section with the following lemma, which provides a class of examples of weights that fulfil Conditions \eqref{second}, \eqref{third} and \eqref{moments}, and whose proof will be given at the end of Section \ref{Sec7}.
}
	\begin{lemma}\label{Lemma implying Conditions}
		\blue{Consider the following choices of weight distributions $\mathscr{L}(\mathscr{W}^{(N)})$:
		\begin{itemize}
			\item[a)] (symmetric Dirichlet-weights)  $\mathscr L(\mathscr W^{(N)})$ is a Dirichlet($\alpha_N,\ldots, \alpha_N$)-distribution with
			\[1+1/\alpha_N = \rho^2+O(1/N)\mbox{ as } N\to \infty.\]
			\item[b)] (Dirichlet-type weights under an exponential moment condition) Let $\mathscr{W}^{(N)}$ be of the form 
			\begin{equation}\label{Dirtype}
			W^{(N)}_i = \frac{Y_i}{\sum_{\ell=1}^{N}Y_\ell} \mbox{ for } i= 1,..., N
			\end{equation}
			 where $Y_1, ..., Y_N$ are independent copies of a non-negative random variable $Y$ for which there exists a $c >0$ such that \quad i)\,  $\EE{\exp(cY)}< \infty$ \quad and \quad ii) \, $\EE{Y^{-c}} < \infty$. 
		\end{itemize}
		Then, in both cases a) and b), the  requirements of Theorem \ref{Haldane  Formel} b) are satisfied.}
		\end{lemma}
\blue{ In \cite{BoGoPoWa2}, Haldane's formula was proved in the case of moderately strong selection for weights as in Lemma \ref{Lemma implying Conditions} b), with the condition $\EE{Y^{-c}} < \infty$ relaxed to $\PP(Y>0) = 1$. There, weights obeying \eqref{Dirtype} were termed  {\em of Dirichlet type}. Obviously, the Wright-Fisher model is a special case, with $Y\equiv 1$.}

	\section{The Cannings  ancestral selection graph. Proof of Theorem \ref{Theorem Duality}}\label{DASG}
	We now define the Cannings  ancestral selection graph, i.e.\ the graph of potential ancestors in a Cannings model with directional selection as announced in Section \ref{secCASP}. The final harvest of this section will be the proof of  Theorem \ref{Theorem Duality}. 
	
	While the branching-coalescing structure of the Moran ancestral selection graph  and the sampling duality stated in Section \ref{Section Moran Model} serve as a conceptual guideline, the ingredients of the graphical construction turn out to be quite different from the Moran case, not least because of the discrete generation scheme.  
	We first describe how, given $\mathscr W^{(g-1)}$ {\em and} the configuration of types of the individuals $(i,g-1)$ in generation $g-1$, the {\em parent} (as well as the type) of an individual $(j,g)$ is constructed from a sequence of  i.i.d.~uniform picks from the unit square.
After this we describe how, given $\mathscr W^{(g-1)}$ (and without prior knowledge of the type configuration in generation $g-1$), the just mentioned i.i.d.~uniform picks from the unit square lead to the {\em potential parents} of an individual $(j,g)$. The latter form a random subset of $[N]\times\{g-1\}$.
	
	To this purpose, as illustrated in Figure \ref{figgamma}, think of the two axes of the unit square as being partitioned in  two respectively $N$ subintervals. The two subintervals that partition  the horizontal unit interval   are  $[0,1-s_N]$ and 
	$(1-s_N,1]$. The $N$ subintervals of the vertical unit interval  have lengths $W_1^{(g-1)}, \ldots, W_N^{(g-1)}$; we call these subintervals $\mathscr I_1^{(g-1)}, \ldots, \mathscr I_N^{(g-1)}$. Let  \[\mathscr B^{(g-1)}:=\{i\in [N]: (i,g-1)\mbox{ is of beneficial type}\}, \, \, \mathscr C^{(g-1)}:=\{i\in [N]: (i,g-1)\mbox{ is of wildtype}\},\]  and define
	\begin{equation}\label{Gjg}
	 \Gamma^{(g-1)} :=  \bigcup_{i\in \mathscr B^{(g-1)}} [0,1]\times \mathscr I_i^{(g-1)}   \cup  \bigcup_{i\in \mathscr C^{(g-1)}} [0,1-s_N]\times \mathscr I_i^{(g-1)}.
	\end{equation} 
\begin{definition} \label{transport} For fixed $j\in [N]$ and $g\in \mathbb Z$, 
		let $U^{(j,g,1)}, U^{(j,g,2)},\ldots$ be a sequence of independent uniform picks from $[0,1]\times [0,1]$ and put
	\begin{equation}
	\gamma(j,g):= \min\{\ell: U^{(j,g,\ell)}  \in \Gamma^{(g-1)}\}.
	\end{equation}	
		Given $ \mathscr B^{(g-1)}, \, \mathscr C^{(g-1)}, \mathscr W^{(g-1)}$ and $U^{(j,g,1)}, U^{(j,g,2)},\ldots$, there is a.s. a~unique $p(j,g)\in [N]$ for which $U^{(j,g,\gamma(j,g))} \in [0,1] \times \mathscr I_{p(j,g)}^{(g-1)}$. The  individual $(p(j,g),g-1)$ is defined to be the \emph{parent} of~$(j,g)$.
	\end{definition}
	Decreeing that individual $(j,g)$ inherits the type of its parent,  we obtain  that a.s.
	\begin{eqnarray} \label{trick1}
	\{(j,g)  \mbox{ is of wildtype}\} &:=& \{(p(j,g),g-1)  \mbox{ is of wildtype}\}\nonumber \\&=& \{U^{(j,g,\gamma(j,g))} \in  [0,1-s_N]\times\bigcup_{i\in \mathscr C^{(g-1)}} \mathscr I_i^{(g-1)}\}
	\end{eqnarray}
	We thus get the transport of $\mathscr C^{(g-1)}$ to the next generation $g$ by putting
\begin{equation}
	\mathscr C^{(g)} := \{j\in [N]: (j,g)\mbox{ is of wildtype}\}.
	\end{equation}
	\begin{remark}\label{CFP}
The $[N]$-valued process $|\mathscr C^{(g)}|$, $g=0,1,\ldots$ is a Cannings frequency process with parameters $N$, $\mathscr L(\mathscr W)$ and $s_N$, as defined in Section \ref{paintsel}. Indeed, 
given $ \mathscr C^{(g-1)}$ and $\mathscr W^{(g-1)}$, the random variables $U^{(j,g,\gamma(j,g))}$, $j=1,\ldots, N$, are independent and uniformly distributed on~$\Gamma^{(g-1)}$, hence \eqref{trick1} and the exchangeability of the components of  $\mathscr W^{(g-1)}$ implies that given $\{|\mathscr C^{(g-1)}|=k\}$ (and with an arbitrary allocation of these $k$ elements in the set $[N]$), the random variable   $|\mathscr C^{(g)}|$ has a mixed Binomial distribution with parameters~$N$ and $P(k, \mathscr W)$ specified by~\eqref{wildtypeprob}. 
\end{remark}	
	Let us now turn to a situation in which the  type configuration of the previous generation is not given, i.e. in which the sets $\mathscr B^{(g-1)}$ and $\mathscr C^{(g-1)}$ and hence also the set $\Gamma^{(g-1)}$ is not know a priori.
	\begin{figure}[ht]
		\centering
		\includegraphics{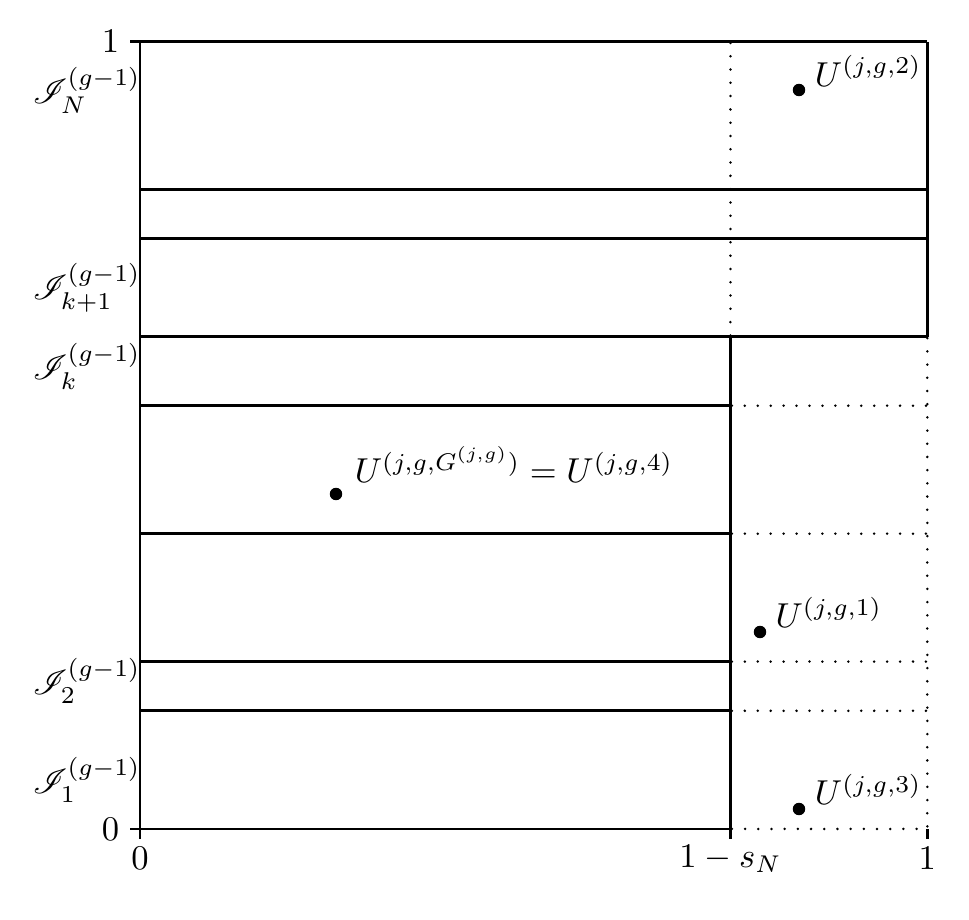}
			\caption {This figure illustrates a case in which  $\mathscr C^{(g-1)} = \{1,\ldots, k\}$,   $\mathscr B^{(g-1)} = \{k+1,\ldots, N\}$, $\gamma(j,g)=2$, $G(j,g)=4$. The individual $(j,g)$ is of beneficial type. Since in this example $\gamma(j,g)$ is strictly smaller than $G(j,g)$, the individual $(j,g)$ must be of beneficial type.} \label{figgamma}
	\end{figure}

	\begin{definition}\label{PASG}
		\begin{itemize}
			\item[i)] For fixed $j\in [N]$ and $g\in \mathbb Z$, 
		let $U^{(j,g,1)}, U^{(j,g,2)},\ldots$ be as in Definition 	\ref{transport}
	 and define
	\begin{equation}\label{geom}
	G(j,g):= \min\{\ell: U^{(j,g,\ell)}  \in [0,1-s_N]\times [0,1]\}.
	\end{equation}
	
			We call $(i,g-1)$ a \emph{potential parent} of $(j,g)$ if\, $U^{(j,g,\ell)} \in [0,1]\times  \mathscr I_i^{(g-1)}$ for some $\ell \le G(j,g)$. Similarly, we call $(i,g-2)$ a \emph{potential grandparent} of $(j,g)$ if $(i,g-2)$ is a potential parent of a potential parent of $(j,g)$. By iteration this extends to the definition of the set $\mathscr A^{(j,g)}_m$ of  \emph{potential ancestors} of  $(j,g)$ in generation $g-m$, $m \ge1$, with  $\mathscr A^{(j,g)}_0:= \{(j,g)\}$.
			\item[ii)]
			For a set $\mathcal J \subset [N]$ and for $m\ge 0$ let  $\mathscr A^{(\mathcal J,g)}_m := \bigcup_{j\in \mathcal J}  \mathscr A^{(j,g)}_m$ 
			be the set of  potential ancestors of $\mathcal J\times \{g\} $  in generation $g-m$.  
			Moreover, let $A^{(\mathcal J,g)}_m:= |\mathscr A^{(\mathcal J,g)}_m|$ 
			be the number of potential ancestors of $\mathcal J\times \{g\}$ in generation $g-m$. 
		\end{itemize}
	\end{definition}
	The  a.s. equality of events asserted in the following lemma is both crucial and elementary.
		\begin{lemma} For $j, g, U^{(j,g,1)}, U^{(j,g,2)}\ldots, \gamma(j,g)$ as in Definition \ref{transport} and $G(j,g)$ as in \eqref{geom},
	\begin{align}
	&\left\{U^{(j,g,\gamma(j,g))} \in [0,1-s_N] \times  \bigcup_{i\in \mathscr C^{(g-1)}}  \mathscr I_i^{(g-1)}\right\}\\
  &\qquad \stackrel{\rm a.s.}{=} \left\{U^{(j,g,1)}, \ldots, U^{(j,g,G(j,g))}\in [0,1]\times  \bigcup_{i\in \mathscr C^{(g-1)}}  \mathscr I_i^{(g-1)}\right\}.\label{trick2}
	\end{align}
	\end{lemma}
	\begin{proof}
To see that the l.h.s.~almost surely implies the r.h.s., consider the first pick that falls into the area $\Gamma^{(g-1)}$ and assume that it  lands in a horizontal stripe belonging to a wildtype individual in generation $g-1$. Then this must be also the first one of the picks that lands in $[0,1-s_N]\times [0,1]$, and no one of the preceding picks could have landed  in a horizontal stripe belonging to a beneficial individual in generation $g-1$. Conversely, to see that the r.h.s.~\eqref{trick2} a.s. implies the l.h.s., consider the first pick that lands in $[0, 1-s_N]\times [0,1]$. If all the picks up  to and including this pick have landed in a horizontal stripe belonging to a wildtype individual in generation $g-1$, then also the first pick that falls into the area $\Gamma^{(g-1)}$ must land in a horizontal stripe belonging to a wildtype individual in generation $g-1$.
\end{proof}
	Combining \eqref{trick1} and \eqref{trick2} with Definition \ref{PASG} we see that for all $g\in \mathbb N$ and all $\mathcal J\subset [N]$
	\begin{equation}\label{formal1}
	\{\mathcal J\subset \mathscr C^{(g)}\} \stackrel{\rm a.s.}{=} \{\mathscr A^{(\mathcal J,g)}_1\subset \mathscr C^{(g-1)}\}
	\end{equation}
	Iterating \eqref{formal1} we arrive at
	\begin{equation}\label{formal2}
	\{\mathcal J\subset \mathscr C^{(g)}\} \stackrel{\rm a.s.}{=} \{\mathscr A^{(\mathcal J,g)}_g\subset \mathscr C^{(0)}\}.
	\end{equation}	
	It is obvious that the random variables $G(j,g)$ defined in \eqref{geom} are independent of the $\mathscr W^{(g')}$, $g'\in \mathbb Z$, and have the property
	\begin{equation}\label{Gjggeom}
	G(j,g), \, g \in \mathbb Z, \, j\in  [N], \quad \mbox{ are independent and Geom}(1-s_N)\mbox{ distributed.} 
	\end{equation}
	This leads directly to the following observation on the number of potential ancestors.
	\begin{remark} \label{dynA}
		Let $g \in \mathbb Z$ and $\mathcal J \subset [N]$ be fixed.
		\begin{itemize}
			\item[i)] The process $|\mathscr A^{(\mathcal J,g)}_m|$, $m = 0, 1,\ldots$, is   a Cannings ancestral selection process (CASP) with parameters $N$, $\mathscr L(\mathscr W)$ and $s_N$, as defined in Section \ref{secCASP}. Indeed, each transition consists of a branching and a coalescence step, where only the latter depends on the $\mathscr W^{(g')}$, \mbox{$g'\in \mathbb Z$}. Specifically,  given $|\mathscr A^{(\mathcal J,g)}_m| = a$, let $H$ have a negative binomial distribution with parameters $a$ and $1-s_N$. Given $H=h$,   $|\mathscr A^{(\mathcal J,g)}_{m+1}|$ is distributed as the number of distinct outcomes in $h$ trials, which given $\mathscr W^{(g-m-1)}$ are independent and follow the  probability weights $\mathscr W^{(g-m-1)}$.
			\item[ii)] For  $m \ge 1$ the exchangeability of the components of $\mathscr W$ implies that,  given $|\mathscr A^{(\mathcal J,g)}_m| = a$, the set $\mathscr A^{(\mathcal J,g)}_m$ is a uniform pick of all subsets of $[N]$ of cardinality~$a$. 
		\end{itemize}
	\end{remark}
	\noindent
	{\em Proof of Theorem \ref{Theorem Duality}}. Let $k$, $n$ and $\mathcal J_n$ be as in Theorem \ref{Theorem Duality}. In  \eqref{formal2} we choose $\mathcal J := \mathcal J_n$ and $\mathscr C_0:= [k]$. Then
	\[\mathbb P(\mathcal J_n \subset [K_g] \mid K_0=k) = \mathbb P(\mathcal J_n\subset \mathscr C^{(g)})= \mathbb P( \mathscr A^{(\mathcal J_n,g)}_g\subset \mathscr C^{(0)}) =        \mathbb P(\mathscr  A_g \subset [k] \mid A_0 = n),\]
where the first equality follows from Remark \ref{CFP}, the second one from \eqref{formal2} and the third one from Remark \ref{dynA}. $\Box$

	Another consequence of~\eqref{formal2} together with Remark \ref{dynA} is the following moment duality, which is interesting in its own right, not least because this was the route through which \cite{GS}  discovered the ``discrete ancestral selection graph'' in the ``quasi Wright-Fisher case'', i.e.\ for $\PP( W_1  = \cdots = W_N)\to 1$.
	\begin{corollary} Let $(K_g)$ and  $(A_m)$ be as in Section \ref{DCA}, let $k, n \in [N]$ and assume that the number of wildtype individuals in generation~$0$ is $k$. Then the probability that a sample of $n$ individuals taken in generation $g \ge 1$ consists of wildtype individuals only is
		\[\mathbb E\left[ \left(\sum_{i=1} ^{K_{g-1}} W_i\right)^{\sum\limits_{j=1}^{n} G^{(j)}} \mid  K_0 = k\right] =  \mathbb E\left[ \left(\sum_{i=1} ^k W_i\right)^{\sum\limits_{j=1}^{A_{g-1}} G^{(j)}} \mid A_0 = n\right],\]
		where $G^{(1)}, G^{(2)}, \ldots$ are independent and Geom$(1-s_N)$-distributed.
	\end{corollary}

 \section{Coupling of the Cannings and Moran ancestral selection processes. \,   Proof of   Theorem \ref{Haldane  Formel}a.}\label{Section DASP}
  In this section we provide a few lemmata preparing the proof of part a) of Theorem \ref{Haldane  Formel}, and conclude with the proof of that part. In particular, in Lemma \ref{Lemma Coupling MASP,CASP} we give a coupling of the Cannings ancestral selection process,  for short CASP, $(A_m)_{m\geq 0}$ defined  in Section \ref{secCASP}  and the Moran ancestral selection process, for short MASP, $(B_r)_{r\geq 0}$ whose jump rates we recalled in Section~\ref{Section Moran Model}.\\
Assume throughout that the $\Delta_N$-valued random weights $\mathscr W^{(N)}=(W_1^{(N)},\dots ,W_N^{(N)})$ fulfil  the Assumptions \eqref{second} and \eqref{third} required in the first part of Theorem \ref{Haldane  Formel}. Let $(s_N)_{N\geq 0}$ be a sequence in $(0,1)$ obeying \eqref{sasymptotics}. Frequently, we will switch to the notation 
\begin{equation}\label{stob}
b_N := -\frac{\ln s_N}{\ln N} \qquad \mbox{ or equivalently to } \qquad s_N = N^{-b_N}
\end{equation} 
with \eqref{sasymptotics} translating into
\[\frac 23 +\eta \le b_N \le 1-\eta.\]
For fixed $N$, and $j\in [N]$ let  
\begin{equation*}
G^{(j)} \mbox{ be independent and Geom}(1-s_N)\mbox{-distributed;}
\end{equation*}
  these will play the role of the random variables $G(j,g)$ defined in \eqref{Gjg}, see also \eqref{Gjggeom}. (Here and whenever there is no danger of confusion, we will suppress the superscripts $N$ and $g$.)

 \begin{lemma}[Moran-like transition probabilities of the CASP]\qquad \label{Lemma jumpsize=1}\\
 	Let $\varepsilon \in (0,\tfrac 16)$. The transition probabilities of the CASP $(A_m)_{m\geq 0}= (A^{(N)}_m)_{m\geq 0}$ obey, uniformly in $k\le N^{1-b_N+\varepsilon}$,
 \begin{align}
 	\PP(A_{m+1}=k\big|A_m=k)&=1-ks_N-\binom{k}{2}  \frac{\rho^2}{N} + O \left(k^4 N^{-2}+k^2s_N^2\right)\label{TransitionProbCASP1} \\
 	\PP(A_{m+1}=k+1\big|A_m=k)&=ks_N + O \left(k^2 s_N^2 +k^4 N^{-2}\right) \label{TransitionProbCASP2}\\
 	\PP(A_{m+1}=k-1\big|A_m=k)&=\binom{k}{2}  \frac{\rho^2}{N} + O \left(k^4 N^{-2}+k^2 s_N^2 \right)\label{TransitionProbCASP3} \\ 	
 	\PP(|A_{m+1}-k| \geq 2\big|A_m=k)&=O(k^4 N^{-2} +k^2 s_N^2 ). \label{TransitionProbCASP4}
 	\end{align}
 \end{lemma}

\begin{remark}
	For $k=2$ we have by \eqref{TransitionProbCASP3} and \eqref{second}
	\begin{align}
\PP(A_{m+1}=k-1|A_m=k)=\EE{\sum_{i=1}^N \frac{W_i^2}{N}}=\frac{\rho^2}{N} + O(N^{-2}),
	\end{align}
	where the first term on the r.h.s.~is the pair coalescence property of the neutral Cannings coalescent with the paintbox $\mathscr W$.
\end{remark}

\begin{proof}[Proof of Lemma \ref{Lemma jumpsize=1}]
Recall that each transition of the CASP consists of a branching and a coalescence step.
	To arrive at the transition probabilities \eqref{TransitionProbCASP1} - \eqref{TransitionProbCASP4} we first estimate the probabilities that $k$ individuals give rise to a total of $k,k+1$ or more than $k+1$ branches and then analyse the probabilities that a single individual is chosen multiple times as a parent. 
	\\
	 Since each individual has a Geom$(1-s_N)$-distributed number of branches,
the probability that $k$ individuals give rise to a total of $k$ branches in the branching step is
 \begin{equation}\label{ProbabilityNoBranching}
 \PP \left( \sum_{j=1}^k G^{(j)}=k\right) = (1-s_N)^k=1-ks_N + O \left(k^2 s_N^2\right) 
 \end{equation}
 and the probability that the individuals give rise to $k+1$ branches is
 \begin{equation}\label{Probability branching size 1}
 \PP \left( \sum_{j=1}^k G^{(j)} =k+1\right)= k (1-s_N)^ks_N = ks_N + O\left(k^2 s_N^2 \right). 
 \end{equation}

 Adding the probabilities in \eqref{ProbabilityNoBranching} and \eqref{Probability branching size 1} yields 
 \begin{align*}
 \PP \left( \sum_{j=1}^k G^{(j)} \geq k+2 \right)=O \left( k^2 s_N^2\right).
 \end{align*}
 
Let us now calculate the probabilities of collisions in a coalescence step, that is the probability that an individual is chosen as a potential parent more than once.
For two branches the pair coalescence probability $c_N$ is given by
\begin{align}
c_N=\EE{\sum_{i=1}^N W_i^2}=\frac{\rho^2}{N}+ O(N^{-2}) \label{pair coalescence prob.}.
\end{align}
In the same manner we obtain the probability for a triple collision as
\begin{align}
d_N=\EE{\sum_{i=1}^N W_i^3}=O(N^{-2}) \label{triple coalescence prob.}.
\end{align}
Using \eqref{pair coalescence prob.} and \eqref{triple coalescence prob.} we control the probability of the event $E$ that there are two or more collisions, with $k$ individuals before the coalescence step. There are two possibilities for this event to occur, either there is at least a triple collision or there are at least two pair collisions. This yields
\begin{align}
\pp{E} \leq \binom{k}{4} \frac{\rho^4}{N^2}+ O\left( \binom{k}{3} N^{-2} \right)  + O \left( k^4 N^{-3}\right)=O\left( k^4 N^{-2} \right). \label{First Moment method}
\end{align}
In order to estimate the probability of having exactly one collision we use the second moment method for the random variable $X=\sum_{i=1}^{k}\sum_{j>i}^{k}X_{i,j}$, where $X_{i,j}=\1_{\{\text{i and j collide}\}}$.
With \eqref{pair coalescence prob.} we get
\begin{align}
\EE{X}=\EE{\sum_{i=1}^{k}\sum_{j>i}^{k}X_{i,j}}=\binom{k}{2} \frac{\rho^2}{N} +O(k^2 N^{-2}). \label{First Moment Jumps}
\end{align}
Furthermore, the second moment of $X$ can be written again due to \eqref{pair coalescence prob.} and \eqref{triple coalescence prob.} as
\begin{align*}
\EE{X^2}&= \EE{\left(\sum_{i=1}^{k}\sum_{j>i}^{k}X_{i,j} \right)^2 } \\ &= \binom{k}{2} \left( \frac{\rho^2}{N}+O(N^{-2}) \right) + O\left( k^3 \EE{X_{1,2} X_{2,3} }\right)+O\left(k^4 \EE{X_{1,2} X_{3,4}}\right) \notag \\
&=\binom{k}{2} \frac{\rho^2}{N}+ O(k^3 N^{-2}) + O(k^4 N^{-2})=\binom{k}{2} \frac{\rho^2}{N} + O(k^4 N^{-2}) 
\end{align*}
This together with \eqref{First Moment Jumps} yields
\begin{align}
\pp{X>0} \geq \frac{\EE{X}^2}{\EE{X^2}} = \frac{\left( \binom{k}{2} \frac{\rho^2 }{N}\right)^2+O\left( k^4N^{-3}\right)}{\binom{k}{2} \frac{\rho^2 }{N} + O(k^4 N^{-2})}= \binom{k}{2} \frac{\rho^2 }{N}(1- O(k^2 N^{-1})),
\end{align}
where the first inequality follows by applying the Cauchy-Schwarz inequality to $X$ and $I_{\{X>0\}}$.
Together with \eqref{First Moment method} we obtain for the random variable $X$ which counts the number of collisions (for $k$ individuals before the coalescence step)
\begin{align*}
&\pp{X=0}=1-\binom{k}{2} \frac{\rho^2 }{N} +O(k^4 N^{-2}) \\
&\pp{X=1}=\binom{k}{2} \frac{\rho^2 }{N} +O(k^4 N^{-2}) \\
&\pp{X\geq 2}=\pp{E} = O(k^4 N^{-2})
\end{align*}
Let $H:= \sum_{j=1}^{A_g} G^{(j)}$.  Then the above calculations allow us to obtain 
\eqref{TransitionProbCASP1}:
 \begin{align*}
 \PP(A_{m+1}=k+1|A_m=k)&= \PP(A_{m+1}= k+1 | A_m =k, H= k) \PP(H=k| A_m = k ) \\
     &\qquad +  \PP(A_{m+1}=k+1| A_m=k, H= k+1) \PP(H=k+1|A_m=k )		\\
 	&\qquad+ 	\PP(A_{m+1}=k+1| A_m=k, H\geq  k+2) \PP(H\geq k+2|A_m=k ) \\
&= \left(1- \binom{k+1}{2} \frac{\rho^2}{N} +O(k^4N^{-2}) \right)(ks_N +O(k^2 s_N^2))+O(k^2s_N^2)\\
&=ks_N +O\left( k^2 s_N^2 +s_N k^5 N^{-2}\right)
 \end{align*}
 
The remaining transition probabilities \eqref{TransitionProbCASP2} - \eqref{TransitionProbCASP4} are derived analogously.
\end{proof}

{\color{black}
Lemma \ref{Lemma CASP comes down} claims that the CASP comes down from $N$ to (the still large state) $N^{1-b+\varepsilon}$ within a time interval of length $o(N^b)$ with a probability that converges quickly to 1 for $N\rightarrow \infty$.   
The proof relies on the following Lemma \ref{WFext}. Therein, we show that with respect to coalescence the ancestral selection process of the Wright-Fisher model is extremal among the CASPs, in the sense that for this process the number $C_k$ of distinct occupied boxes after the coalescing half-step is stochastically the largest.}

\begin{lemma}\label{WFext}
For natural numbers $N$ and $k$ let $Z_1, Z_2 \ldots, $ be i.i.d. $[N]$-valued random variables with $w_i:= \mathbb P(Z_1=i)$, $i\in [N]$. Then for each $k\in \mathbb N$ the random variable $C_k := |\{Z_1, \ldots, Z_k\}|$ is stochastically largest for $w_1=\cdots  = w_N=\tfrac 1N$. 
\end{lemma}
\begin{proof}
{\color{black}We interpret the random variables $C_k$ in terms of the ``coupon collector's problem''.} For $\ell \in [N]$ let $T_\ell := \min\{k: C_k \ge \ell\}$. Then we have the obvious identity
\[\{C_k \ge \ell\}= \{T_\ell \le k\}.\]
\cite[Theorem 2]{ABBS} states that $\mathbb P(T_\ell \le k)$ is largest for $w_1=\cdots  = w_N=\tfrac 1N$.
\end{proof}

The quantities $b$, $A_m$, $\tau$ appearing in the next lemma all depend on $N$; we will suppress this dependence in the notation.
\begin{lemma}[CASP coming down from huge to large] \label{Lemma CASP comes down}
	Let $(A_m)_{m \in \N_0}$ be a CASP, $0<\varepsilon <\frac{2}{3}$, $A_0 =N$ and denote by $\tau=\inf \left\lbrace m \geq 0 : A_m \leq N^{1-b+\varepsilon} \right\rbrace $ the first time the CASP crosses the level $N^{1-b +\varepsilon}$. Then there exists a $\delta >0$, such that for any constant $c >0$ 
	\begin{align}
	\PP\left( \frac{\tau}{N^b} >c\right)= O(\exp(- N^{\delta})). \label{equation coming down}
	\end{align}
\end{lemma}

\begin{proof}
	The branching step of the CASP dynamics only depends on $s_N$ and neither on the distribution  nor the realization of $\mathscr W$, see Remark \ref{dynA} i). On the other hand, the coalescence step only depends on $\mathscr W$\!. By Lemma \ref{WFext} the size of this coalescence step is stochastically lower bounded by the corresponding step in an ancestral selection process of the Wright-Fisher model. 
	Thus, among all CASP's  with selective strength $s_N$, the CASP of the Wright-Fisher model with selection is the slowest to come down from $N$ to $N^{1-b+\varepsilon}$; therefore  we use the stopping time corresponding to the Wright-Fisher model as a stochastic upper bound for $\tau$. Consequently, we assume for the rest of the proof that $\mathscr{W} = (\frac{1}{N}, ..., \frac{1}{N})$.\\
	To show \eqref{equation coming down} we estimate $\EE{A_{m+1}|A_m=k}$ for $1 \leq k \leq N$.\\
	$A_{m+1}$ denotes the number of potential parents of $A_m$ individuals, that is 
	\begin{align*}A_{m+1} =\sum_{i=1}^N \1_{ \{ \text{ Individual } i \text{ is a potential parent of some of the } A_m \text{ individuals} \}}.\end{align*} 
	Let $H=\sum_{j=1}^{A_m} G^{(j)}$, with $G^{(j)}\sim \text{Geom}(1-s_N)$ and independent for $j \in [N]$. Then
	\begin{align*}
	\PP(\text{Individual $i$ is chosen as a potential parent}|A_m)=1-\EE{\left( 1- \frac{1}{N}\right)^H\Big|A_m}
	\end{align*} 
	for $i\in [N]$. Hence, for $k \geq 1$ and $x = \frac{1}{N(1-s_N)}$
	\begin{align}
	\EE{A_{m+1}|A_m=k}&=N\EE{\left(1-\left(1-\frac{1}{N}\right)^{H}\right)\Big{|}A_m=k} \notag \\
	&= N \left(1 - \left(1-\frac{1}{N(1-s_N + \frac{s_N}{N})}\right)^k\right) \label{probabnegbin} \\
	&\leq - N\left( k \ln(1-x) + \frac{(k \ln( 1-x))^2}{2} + 
	\frac{(k \ln( 1-x))^2 k \ln( 1-x) }{6}\right) \label{taylor} \\
	&\leq N\left( k x - \frac{(k x)^2 (1 -2 s_N)}{3} \right) \notag \\
	&= \frac{k}{1-s_N} -  \frac{k^2 (1-2s_N)}{3 N } \label{estimate expectation A_m} ,
	\end{align}
	where for \eqref{probabnegbin} we use the probability generating function of the  negative binomial distribution and for \eqref{taylor} we use an estimate for the remainder of the corresponding Taylor expansion.
	Let $0 < \eps' < \eps$. From \eqref{estimate expectation A_m} follows
	\begin{align*}
	\EE{A_{m+1} |A_m} &\leq \max \left\{ \frac{A_m}{1 -s_N} - \frac{A_m N^{1-b+\eps'} (1 -2s_N) }{ 3 N}, N^{1-b+\eps'} \right\} \\
	&=\max \{q_N A_m,N^{1-b+\eps'}\},
	\end{align*}
	with $q_N=\frac{1}{1-s_N} - \frac{N^{1-b+\eps'} (1-2s_N)}{3N}$. This yields 
	\begin{align}
	\EE{A_m|A_0=N} \leq \max \{ q_N^m N,N^{1-b+\eps'}\}.
	\end{align}
	For any $m \geq c_1 N^{b-\eps'} \ln N$, for some appropriate constant $c_1>0$ we have thus the estimate $\EE{A_m|A_0=N}\leq N^{1-b+\eps'}$.
	By Markov's inequality we obtain
	\begin{align}
	\PP( A_{c_1N^{b-\eps'} \ln N} > N^{1-b+\eps})\leq N^{\eps'-\eps}\to 0 
	\end{align}
	as $N \to \infty$. If $(A_m)_{m\geq0}$ did not reach $N^{1-b+\eps}$ after $c_1N^{b-\eps'} \ln N$ steps we can start the process in $N$ again and wait another $c_1N^{b-\eps'} \ln N$ steps and check whether the process did reach the level $N^{1-b+\eps}$.  By using this argument $N^{\delta_1}$ times this yields, for any $0 < \delta_1 < \eps'$, the following upper bound for the probability to stay above $N^{1-b+\eps}$ for the generations $m \le c_1  N^{b-\eps' +\delta_1}$:
	\begin{align*}
	\PP(A_m > N^{1-b+\eps} \text{ for } m \in \{0,...,c_1  N^{b-\eps' +\delta_1} \ln N \}  )  & \leq \PP(A_{c_1N^{b-\eps'} \ln N} > N^{1-b+\eps} )^{N^{\delta_1}} \\& \leq (N^{\eps'-\eps})^{N^{\delta_1}}.
	\end{align*}
	Since $\eps>\eps'$ and $N^{b-\eps' +\delta_1} < N^b$, we have  $(N^{\eps'-\eps})^{N^{\delta_1}} = O(\exp(-N^{\delta}))$ for some appropriate $\delta>0$ from which the assertion follows. 
\end{proof}
\noindent

\noindent From Lemma \ref{Lemma CASP comes down} we obtain the following corollary:
\begin{corollary}\label{Corollary Expectation DASG}
 	Let $(A_m)_{m \geq 0}$ be a CASP. 
	Then for any $m_0\ge 0$ there exists a $C>0$ such that for all $N\ge 1$ and  all $j \ge N^{1-b+\varepsilon}$ 
\begin{align*}
&\pp{A_{m_0+N^{b}} >j|A_0=N} \le C N^{1-b+\eps}/j,\\
&\EE{A_{m_0+N^{b}}|A_0=N}= O(\ln(N) N^{1-b+\eps}).
\end{align*}
 \end{corollary}
\begin{proof}
For simplicity assume that $m_0=0$, but the same proof works for any $m_0\in \N$ as $\pp{A_{m_0+N^b}>j|A_0=N} \leq \pp{A_{m_0+N^b}>j|A_{m_0}=N} $. Due to Lemma \ref{Lemma CASP comes down} for the stopping time $\tau=\inf \{m \geq 0: A_m \leq N^{1-b+\eps} \} $ it holds $\pp{\tau>N^{b}}= O(\exp(-N^{-\delta}))$, with $\delta $ as in Lemma \ref{Lemma CASP comes down}. By Lemma \ref{Lemma jumpsize=1} we can compare the jump probabilities and obtain that there exists some $x_0\leq N^{1-b+\eps/2}$ such that above $x_0$ the upward drift  is smaller than the downward drift. This yields that the process stopped in $x_0$ is a supermartingale. Consequently since $x_0< N^{1-b+\eps}$, we have for any $m' \in \N$ by the strong Markov property
\begin{align}
	\EE{A_{\tau+m'}} \leq N^{1-b+\eps}.
\end{align}
	Hence by Markov's inequality we obtain
\begin{align*}
	\pp{A_{N^b} >j|A_0=N}&\leq \pp{A_{N^{b}} >j|A_0=N,\tau \leq N^{b}}+\pp{\tau> N^b} \\
	&\leq\frac{\EE{A_{N^{b}}|A_0=N,\tau \leq N^{b} } }{j}+O(\exp(-N^{\delta }))\\
	& \leq \frac{\EE{A_\tau}}{j}+O(\exp(-N^{\delta})), 
\end{align*}
which shows the first part. For the second part observe that
\begin{align*}
&\EE{A_{N^{b}}|A_0=N} =\sum_{j=1}^{N} \pp{A_{N^{b}}>j|A_0=N} \\
&=\sum_{j=1}^{N^{1-b+\eps}} \pp{A_{N^{b}}>j|A_0=N}+\sum_{j=N^{1+b-\eps}}^{N} \pp{A_{N^{b}}>j|A_0=N} \\
&\leq N^{1-b+\eps} +\sum_{j=N^{1+b-\eps}}^{N} \frac{N^{1-b+\eps}}{j} = O( N^{1-b+\eps} \ln N).
\end{align*}
\end{proof}
The following three lemmata provide some properties about the Moran process and the coupling of a Moran process to a Moran process in stationarity. For the remainder of this section we will fix three constants
\begin{align}
\delta_1 \in (0,1), \, \, 0<\delta_3 < \delta_2/2. \label{Definition constants}
\end{align}
The role of $\delta_1$ will be to specify a region $\left[\frac{2s_N}{2s_N + \rho^2 }N (1-\delta_1),\frac{2s_N}{2s_N + \rho^2 }N (1+\delta_1) \right]$ around MASP's center of attraction. The constant $\delta_2$ will appear in factors $N^{\delta_2}$ that stretch some time intervals, and the constant $\delta_3$ will be an exponent in small probabilities $O(\exp(-N^{\delta_3}))$.
\begin{lemma}[MASP's hitting time of the central region] \label{Lemma MASP comes down from N,up from 1}
	Let $(B_r)_{r\geq 0}$ be a MASP started in some state $n\in[N]$ and let $T=\inf\{ r\geq 0:B_r \in [\frac{2s_N}{2s_N + \rho^2 }N (1-\delta_1),\frac{2s_N}{2s_N + \rho^2 }N (1+\delta_1) ] \}$. Then,
	\begin{align}
	\pp{T\leq N^{b+\delta_2}|B_0=n} =  1 - O(\exp( -N^{\delta_3})).
	\end{align}
\end{lemma}
\begin{proof}
	We proceed in a similar manner as \cite{PP} and separate the proof into two cases
	\begin{itemize}
		\item[i)] $B_0 > \frac{2s_N}{2s_N + \rho^2 }N (1+\delta_1)$
		\item[ii)] $B_0<\frac{2s_N}{2s_N + \rho^2 }N (1-\delta_1)$.
	\end{itemize}
	For case i) the proof relies on a stochastic domination of the MASP by a birth-death process, while for case ii) we construct a pure birth process  that is stochastically dominated by the MASP. We start by proving case i).\\
	Assume the most extremal starting point $B_0=N$. 
	We couple the process $(B_r)_{r\geq 0}$ with a birth-death process $(\overline{B}_r)_{r\geq 0}$ which stochastically dominates $(B_r)_{r\geq 0}$ until $(B_r)_{r\geq 0}$ crosses the level $\frac{2s_N}{2s_N + \rho^2 }N (1+\delta_1)$. $(\overline{B}_r)_{r\geq 0}$ is defined as the Markov process with state space $\N_0$ and the following transition rates
	\begin{itemize}
		\item $k \to k+1$ with rate $k s_N=:\overline{\beta} k$
		\item $k \to k-1$ with rate $k \frac{s_N \rho^2 }{2s_N+\rho^2 } (1+\delta_1)=:\overline{\alpha} k$.
	\end{itemize}
Note that $\overline{\beta}k\geq s_N k(N-k)/N $ and $\overline{\alpha} k \leq \binom{k}{2} \frac{\rho^2}{N}$ for any $k \geq \frac{2 s_N}{2 s_N + \rho^2} N (1+\delta_1)$. Hence, we can couple $(\overline{B}_r)_{r\geq 0}$ and $(B_r)_{r\geq 0}$ such that $B_r \leq \overline{B}_r$ a.s. as long as $B_r \geq \frac{2 s_N}{2 s_N + \rho^2} N (1+\delta_1)$. In particular, we have 
\begin{align}\PP(T \geq r| B_0=k) \leq \PP(\overline{\tau}_0 \geq r| \overline{B}_0=k) \label{approx by branching}
\end{align}
 when we set $\overline{\tau}_0:=\inf \left\lbrace r \geq 0 : \overline{B}_r=0 \right\rbrace$ and $k \geq \frac{2 s_N}{2 s_N + \rho^2} N (1+\delta_1)$.
 For the birth-death process $\overline{B}_r$ we can estimate  $\overline{\tau}_0$,  by a classical \textit{first step} analysis
	\begin{align*}
	\PP(\overline{\tau}_0 \geq r|\overline{B}_0=1 )=  &  (1-(\overline{\alpha}+\overline{\beta} )dr) \,\PP(\overline{\tau}_0 \geq r-dr|\overline{B}_0=1) \\ &+ \overline{\beta} dr(1-(1-\PP(\overline{\tau}_0 \geq r-dr|\overline{B}_0=1))^2) 
	\end{align*}
    Setting $f(r)=\PP(\overline{\tau}_0 \geq r|\overline{B}_0=1)$ we obtain
 \begin{align}    
	f^{\prime}(r)&= ( \overline{\beta} -\overline{\alpha}) f(r)-\overline{\beta} f(r)^2 \label{Diffeq2} \notag
	\end{align}
	with $f(0)=1$ which is solved by
	\begin{align*}
	f(r)= \frac{\overline{\alpha}-\overline{\beta} }{\overline{\alpha} e^{r (\overline{\alpha}-\overline{\beta})} -\overline{\beta}}.
	\end{align*}
	Observe that $\overline{\alpha}-\overline{\beta}=\frac{s_N \rho^2 }{\rho^2 +2s_N} (1+\delta_1)-s_N =\delta_1 s_N (1+o(1))$, hence
	\begin{align}
	f(N^{b+\delta_2})&=\frac{\delta_1 s_N(1+o(1))}{ \frac{\rho^2 s_N}{2s_N+\rho^2}(1+\delta_1) \exp \left( N^{b+\delta_2} \delta _1 s_N  (1+o(1))  \right)-s_N}\\
	&=\frac{\delta_1 (1+o(1))}{\frac{\rho^2}{2s_N+\rho^2}(1+\delta_1) \exp \left( N^{\delta_2} \delta _1(1+o(1))  \right)-1}. \label{one indi}
	\end{align}

    From \eqref{approx by branching} and \eqref{one indi} we finally estimate
		\begin{align*}
	\PP(T<N^{b+\delta_2} |B_0=N) &\geq \PP(\overline{\tau}_0<N^{b+\delta_2} |\overline{B}_0=N)\geq 
	1-  N f(N^{b+\delta_2})\\
	&=1- O(N \exp(-N^{\delta_2})=1- O( \exp(-N^{\delta_3}))
	\end{align*}
	for any $\delta_3<\delta_2.$ This proves part i).\\\\
	Now it remains to prove case ii). Again assume the most extremal starting point  $B_0=1$. 
	Let $(\underline{B}_r)_{r\geq 0}$ be a birth-death process which jumps
	\begin{itemize}
		\item from $k$ to $k+1$ at rate $ks_N (1-\frac{2s_N}{2s_N+ \rho^2}(1-\delta_1))=:\underline{\beta} k$
		\item from $k$ to $k-1$ at rate $k \frac{s_N \rho^2 }{2 s_N+\rho^2 }(1-\delta_1)=: \underline{\alpha} k$.
	\end{itemize}
	Observe that $\underline{\beta} k \leq s_N k (N-k)/N$ and $\underline{\alpha} k  \geq \binom k 2 \frac{\rho^2}{N}$ as long as $k \leq \frac{ 2 s_N}{\rho^2 + 2 s_N} N (1-\delta_1)$. Hence, we can couple 
	$(\underline{B}_r)_{r\geq 0}$ and $(B_r)_{r\geq 0}$ such that $B_r \geq \underline{B}_r$ as long as $B_r \leq \frac{ 2 s_N}{\rho^2 + 2 s_N} N (1-\delta_1)$.
	
	The extinction probability $\xi_0$ of $(\underline{B}_r)_{r\geq 0}$ is the smallest solution of
	\begin{align*}
	\xi= \frac{\underline{\beta}}{\underline{\beta} +\underline{\alpha}} \xi^2 + \frac{\underline{\alpha}}{\underline{\beta} +\underline{\alpha}},
	\end{align*}
	that is $\xi_0=\frac{\underline{\alpha}}{\underline{\beta}}<1$. Let $(\underline{B}_r^{I})_{r\geq 0}$ be the pure birth process consisting of the immortal lines of $(\underline{B}_r)_{r\geq 0}$, i.e.\ each line branches at rate $(1-\xi_0) \underline{\beta}$.\\
	Let $ \tau=\inf \{ r \geq 0 : B_r \geq \frac{2}{ \rho^2 } s_N N (1-\delta_1) \}$ be the time when $(B_r)_{ r\geq 0}$ 
	reaches the level $\frac{2}{ \rho^2 } s_N N (1-\delta_1)$ and define $\underline{\tau}^{I}$ and $\underline{\tau}$ in the same way for the processes $(\underline{B}_r^{I})_{t \geq 0}$ and $(\underline{B}_r)_{r \geq 0}$ respectively in place of $(B_r)_{r \geq 0}$, then $\underline{\tau}^{I}\geq \underline{\tau} \geq \tau$ a.s. In order to prove ii) it remains to show $\PP (\underline{\tau}^{I} \geq N^{b+\delta_2 })=O(\exp(-N^{\delta_3}))$ for $\delta_3 >0$. We have	
	\begin{align}
	\EE{\underline{\tau}^I}&=\EE{\sum_{i=1}^{ \left\lfloor\frac{2s_N}{2s_N+\rho^2} N(1-\delta_1)\right\rfloor} \frac{1}{i \underline{\beta} (1-\xi_0)} }= \frac{1}{\underline{\beta} (1-\xi_0)} \left(  \ln\left( \frac{2s_N(1-\delta_1)}{2s_N+\rho^2 } N \right)+O(1) \right)\notag \\
	&=\frac{1}{\delta_1 s_N} \left(  \ln\left( \frac{2s_N(1-\delta_1)}{2s_N+\rho^2 } N \right)+O(1) \right)= \frac{1}{\delta_1} N^b \ln\left( \frac{2(1-\delta_1)}{\rho^2 } N^{1-b} \right)(1+O(s_N)) \notag \\
	&=\frac{1-b}{\delta_1} N^b \ln\left( N \right)(1+O((\ln N)^{-1})).
	\end{align}
	We can estimate
	\begin{align}
		\PP(\underline{\tau}^I>N^{b+\delta_2}|\underline{B}^I_0=1)\leq \PP_1(\underline{\tau}^I>N^{b+\delta_2/2}|\underline{B}^I_0=1) ^{N^{\delta_2/2}}
	\end{align}	
for $\delta_2>0$ by separating the time interval of length $N^{b+\delta_2}$ into $N^{\delta_2/2}$ time intervals of length $N^{b+\delta_2/2}$ and realizing that if $(\underline{B}_r^I)_{t\geq 0}$ did not reach the level $\frac{2}{ \rho^2 } s_N N (1-\delta_1)$ in a time interval of length $N^{b+\delta_2/2}$ then in the worst case $(\underline{B}_r)_{r\geq 0}$ is $1$ at the start of each time interval. 
	
	By Markov's inequality we then arrive at
	\begin{align}
	\PP(\underline{\tau}^I \geq N^{b+\delta_2})&\leq \PP(\underline{\tau}^I >N^{b+\frac{\delta_2}{2}} )^{N^{\frac{\delta_2}{2}}} \notag \leq \left( \frac{1}{\delta_1} N^{-\frac{\delta_2}{2}} \ln N \right)^{N^{\frac{\delta_2}{2}}}\\
	&=\exp \ln \left( \left( \frac{1}{\delta_1} N^{-\frac{\delta_2}{2}} \ln N \right)^{N^{\frac{\delta_2}{2}}} \right)\notag \\
	& \leq \exp\left( -\frac{\delta_2}{2} N^{\frac{\delta_2}{2}} \right)
	=O(\exp(-N^{\delta_3})) \label{Estimated Probability}
	\end{align}
	for $\delta_3<\delta_2/2$. From \eqref{Estimated Probability} we can directly conclude $\PP(\tau \geq N^{b+\delta_2})= O( \exp( -N^{\delta_3}))$, which together with part i) finishes the proof.
\end{proof}

\begin{lemma}[MASP's leaving time of the central region] \label{Lemma MASP stays close to expectation}
Let $(B_r)_{r\geq 0}$ be a MASP started in $x \in [\frac{2s_N}{2s_N+\rho^2} N(1-\delta_1),\frac{2s_N}{2s_N+\rho^2} N(1+\delta_1)]$ and assume in addition to \eqref{Definition constants} that $0< \delta_1< \frac{1}{2}$ and $0<\delta_2 <\frac{\eta}{3}$. Let $S=\inf \{r\geq 0 : B_r \notin [\frac{2s_N}{2s_N+\rho^2} N(1-2\delta_1),\frac{2s_N}{2s_N+\rho^2} N(1+ 2\delta_1) ]\}$. Then  
	\begin{align}\label{Equation stay close to expectation}
	\pp{S>N^{b+\delta_2} } \geq 1- o(\exp(- N^{1-b-3\delta_2})).
	\end{align}
\end{lemma}
\begin{proof}
	Assume we have $B_0 \in [\frac{2s_N}{2s_N+\rho^2 }N (1-\delta_1),\frac{2s_N}{2s_N+\rho^2 } N (1+\delta_1)]$. To prove 
	\eqref{Equation stay close to expectation} we couple $(B_r)_{r\geq 0}$ with a symmetric (discrete time) random walk $(S_n)_{n\geq 0}$, and thus ignore the drift to  $\frac{2s_N}{2s_N+\rho^2 }N$. An application of Theorem 5.1 iii) of \cite{SJ} yields that $(B_r)_{r\geq0}$ makes at most $N^{1-b +2\delta_2}$ many jumps in a time interval of length $N^{b+\delta_2}$ with probability $1-O(\exp(-N^{1-b+2 \delta_2})$, see also the estimate \eqref{Exp bound exp variables} in Lemma \ref{Lemma Coupling MASP,CASP} below, where we analyse the jumps and jump times of the MASP in more detail. Hence,	
		\begin{align}
		&\PP_x \left( B_r \notin [\frac{2s_N}{2s_N+\rho^2 }N (1-2\delta_1),\frac{2s_N}{2s_N+\rho^2 } N (1+2\delta_1)] \text{ for some } r \leq N^{b+\delta_2} \right) \notag 
		\\ \leq &\PP_0 \left( S_n \notin [-\delta_1 \frac{2s_N}{2s_N+\rho^2 }N ,\delta_1\frac{2s_N}{2s_N+\rho^2 } N]\text{ for some } n \leq N^{1-b+2\delta_2} \right)\notag \\
		=&2\PP_0 \left( \max_{1\leq n \leq N^{1-b+2\delta_2}}S_n \notin [0,\delta_1\frac{2s_N}{2s_N+\rho^2 } N]\right) \notag \\
		=&4\PP_0 \left( S_{N^{1-b+2\delta_2}} >\delta_1\frac{2s_N}{2s_N+\rho^2 } N \right)  \label{Reflection principle}\\
		\leq& 4 \exp \left(-c N^{1-b-2\delta_2} \right)=o(\exp( -N^{1-b-3\delta_2}))
		\label{Hoeffding} 
		\end{align}
		for some appropriate $c>0$ independent of $N$. To obtain equation \eqref{Reflection principle} and inequality \eqref{Hoeffding} we used the reflection principle and Hoeffding's inequality.  This finishes the proof.
	\end{proof}
 \begin{lemma}[MASP close to stationarity] \label{Lemma Moran coupling Stationarity}
 Let	$(B_r)_{r\geq 0}$ be a MASP started in $k$ individuals, with $1 \leq k \leq N$, then 
 	\begin{align*}
d_{\rm TV}(\mathscr L(B_{N^{b+\delta_2}}), \mathscr L(B_{\rm eq})) = O(\exp(-N^{\delta_3}))
 	\end{align*}
 	with $B_{\rm eq}=B_{\rm eq}^{(N)}$ as in \eqref{fixNMoran}, i.e.\ distributed as a Binomial$(N, \frac{2s_N}{2s_N + \rho^2})$-random variable  conditioned to be strictly positive, and the constant in the Landau $O$ is uniform in $k$.	
 \end{lemma}
\begin{proof}
    We follow a similar strategy as the one used in the proof of Lemma 2.10 in \cite{PP}.
    Let $(B_r^{eq})_{r\geq 0}$ be a MASP started in the stationary distribution. Assume that in the graphical representation at time 0 either the lines of $B_0$ are contained in $B^{eq}_0$ or vice versa. Then $B_r \leq B^{eq}_r$, for all $r \geq 0$, or vice versa $B^{eq}_r \leq B_r$. Then $\PP(B_{N^{b+\delta_2}}=k) =  \PP(B^{eq} =k)(1-O(e^{-N^{\delta_3}}))$ follows, once we show that at time $N^{b+\delta_2}$ both processes are equal with probability $(1-O(e^{-N^{\delta_3}}))$.
    
    The tuple $(B^{eq}_r, B_r)_{r \geq 0}$, and the tuple $(B_r, B^{eq}_r)_{r\geq 0}$ resp., have the following transition rates: jumps from  $(k, \ell)$ for $1 \leq k \leq \ell \leq N $ to
		 \begin{itemize}
              \item $(k+1, \ell +1)$ occur at rate $s_N k(1-\frac{\ell}{N})$
              \item $(k, \ell +1)$ occur at rate $s_N(\ell -k)(1- \frac{\ell}{N})$
		      \item $(k+1,\ell)$ occur at rate $ks_N \frac{\ell-k}{N}$
		      \item $(k, \ell-1)$ occur at rate $\frac{\rho^2}{N} \left( \binom{\ell-k}{2} +  (\ell -k) k \right)$
		      \item $(k-1, \ell -1)$ occur at rate $\frac{\rho^2}{N} \binom k 2$.
		      \end{itemize} 

    To proceed further we consider the two cases
	\begin{itemize}
		\item[i)] $B_0 > B^{eq}_0$
		\item[ii)] $B_0 < B^{eq}_0$
	\end{itemize}
	separately.

We begin with Case i). 
Consider the process $(Z_r)_{r\geq 0}$ defined as $Z_r: = B_r - B^{eq}_r$ and condition on the two events that the process $B^{eq}_0$ is started in a state in $[\frac{2s_N}{2s_N+\rho^2} N(1-\delta_1),\frac{2s_N}{2s_N+\rho^2} N(1+\delta_1) ]$ and stays in 
$[\frac{2s_N}{2s_N+\rho^2} N(1-2\delta_1),\frac{2s_N}{2s_N+\rho^2} N(1+2\delta_1) ]$ for some $0<\delta_1<\frac{1}{2}$. The probability of each event can be estimated by $1 -O(\exp(- N^{\delta_2}))$,  the former event by Hoeffding's inequality  and the latter with Lemma \ref{Lemma MASP stays close to expectation}.
 The process $(Z_r)_{r\geq 0}$ jumps from $z$ to $z+1$ at most at rate $s_n  z$ and under the above condition  $(Z_r)_{r\geq 0}$ jumps from $z$ to $z-1$ at least at rate $ \rho^2\frac{2s_N}{2s_N+\rho^2}(1 -2 \delta_1) z$: 
		      If $(Z_r, B_r, B^{eq}_r)= (z,\ell, k)$
		      jumps to $(z-1, \ell-1,k)$ occur at rate $\frac{\rho^2}{N}(\binom z 2 + z k)$ and jumps to $(z-1,\ell,k+1)$ at rate $ks_N \frac{\ell-k}{N}$. Therefore, the process $(Z_r)_{r\geq 0}$ jumps from $z \to z-1$ at rate $r_{z,z-1}=\frac{\rho^2}{N}(\binom z 2 + z k)+ks_N \frac{z}{N}$.
              Due to the condition and the assumption that $\ell \geq k \geq \frac{2s_N}{2s_N+\rho^2} N (1 -2 \delta_1)$ we can bound 
      \begin{align*}
              r_{z,z-1} & = \frac{\rho^2}{N}\left(\binom z 2 + z k\right)+ks_N \frac{z}{N}\geq  \frac{\rho^2}{2N}z(k+\ell -1)+z \frac{2s_N^2}{2s_N+\rho^2}(1 -2 \delta_1)
              \\ &
              \geq z\frac{\rho^2}{2N} 2 \frac{2s_N}{2s_N+\rho^2} N (1 -2 \delta_1) = z \rho^2\frac{2s_N}{2s_N+\rho^2}(1 -2 \delta_1).
        \end{align*}
		      Hence, we can couple $(Z_r)_{r \geq 0}$ to a birth-death process $(Z'_r)_{r\geq 0}$ with individual birth rate $s_N=:\beta'$ and individual death rate $\rho^2\frac{2s_N}{2s_N+\rho^2}(1 -2 \delta_1)=:\alpha'$, such that $Z_r \leq Z'_r$ a.s.
		      Let $\xi:=\inf \{r\geq 0 : Z_r=0 \}$ and $\xi':=\inf \{r\geq 0 : Z_r'=0 \}$. Obviously it holds $\PP(\xi \geq r) \leq \PP(\xi ' \geq r)$ for all $r \geq 0$.
		      As in the proof of Lemma \ref{Lemma MASP comes down from N,up from 1} we estimate 
		      \begin{align*}
		      \PP(\xi' \geq N^{b+\delta_2}| Z'_0=1) 
		      &=\frac{\left(\frac{2\rho^2 (1-2\delta_1)}{2s_N+\rho^2}-1\right) s_N}{\left(\frac{2\rho^2 (1-2\delta_1)}{2s_N+\rho^2}-1\right) s_N \exp(\left(\frac{2\rho^2 (1-2\delta_1)}{2s_N+\rho^2}-1\right) N^{\delta_2}) -s_N} \\ &
		      =O(\exp(-c_N N^{\delta_2}))
		      \end{align*}
		      with $c_N=\left(\frac{2\rho^2 (1-2\delta_1)}{2s_N+\rho^2}-1\right)\to 2(1-2\delta_1)-1>0$. Since $Z_0 \leq N$ the probability that all lines go extinct before time $N^{b+\delta_2}$ can be estimated by
		      \begin{align*}
		      \PP(Z_{N^{b+\delta_2}}=0)&\geq \left( 1-\exp(-c_N N^{\delta_2})\right)^{N} = 1 - O \left( \exp( -N^{\delta_3})\right),
		      \end{align*}
            which proves Lemma \ref{Lemma Moran coupling Stationarity} in Case i).\\\\
		      In Case ii) we first wait until $(B_r)_{r\geq 0}$ reaches the level $\frac{2}{\rho^2} s_N N(1-\delta_1)$ within a time interval of length $O(N^{b+\delta_2})$ with probability $1- O(\exp(-N^{\delta_3}))$ due to Lemma \ref{Lemma MASP comes down from N,up from 1} and we assume that $B^{eq}_0$ is started in at least $\frac{2}{\rho^2} s_N N(1-\delta_1)$, which happens with probability $1-O(\exp(-\delta_1^2 N))$ due to Hoeffding's inequality. Then due to Lemma \ref{Lemma MASP stays close to expectation} both processes remain bounded from below by $\frac{2}{\rho^2} s_N N(1-2\delta_1)$. When $(B_r)_{r\geq 0}$ has reached at least the level $\frac{2}{\rho^2} s_N N(1-\delta_1)$) consider $Z_r = B_r - B^{eq}_r$. Then the same arguments as in Case i) show the claim.    
\end{proof}

As mentioned in the sketch of proof of Theorem \ref{Haldane  Formel} in Section \ref{Main Results} we aim to couple the CASP with the MASP. We have seen in the calculations before that in the regime where the number of potential ancestors is at most of order $N^{1-b+\varepsilon}$ for $\varepsilon$ sufficiently small the transition probabilities of these two processes are essentially the same for a time interval of length of order $O(N^{b+\eps})$. In particular in a time interval of length $O(N^{b+\eps})$ we can exclude jumps of size $2$ or bigger in the CASP with probability $O(N^{-\delta})$).\\
\begin{lemma} [Coupling of MASP and CASP] \label{Lemma Coupling MASP,CASP}
	Let $0< \varepsilon< \frac{\eta}{2}$, and $0<\delta=3\eta -6 \eps$. 
	There exists a coupling of the MASP $(B_r)_{r \geq 0}$ and the CASP  $(A_m)_{m \geq 0}$ such that for all common initial values $k_0$ with $1 \leq k_0 \leq N^{1-b+\varepsilon}$
	\begin{align}
	\PP \left( |A_m-B_m| \leq 1,\forall m \in \{0,\dots , N^{b+\eps}\} \right)=1- O(N^{-\delta}). \label{Coupling CASP ASG}
	\end{align}
	with the constant in the Landau $O$ uniform in $k_0$.
\end{lemma}
\begin{proof}	
	Let $A_0=B_0=k_0 \leq N^{1-b+\eps}$. We will show that the CASP and the MASP can be coupled such that the jump times of the CASP and the MASP occur consecutively with probability $1-O(N^{-\delta})$. Since the transition probabilities of the CASP and the MASP are essentially the same we can also couple the jump directions with high probability.
	To show that the jump times occur consecutively we first show the following claim.\\
	Claim 1: With probability $1- O(N^{-\delta})$ 
	the MASP makes in each of the time intervals $[\ell-1, \ell]$ at most one jump.\\
	By Lemma \ref{Lemma MASP comes down from N,up from 1} and \ref{Lemma MASP stays close to expectation}  the MASP stays below $2 N^{1-b+\eps}$ with probability $1-O(\exp(-N^{\delta_3}))$.\\
	Denote by $r_{k,k+1}$ and $r_{k,k-1}$ the jump rates for the MASP from $k$ to $k+1$ and from $k$ to $k-1$ respectively with $\gamma=\rho^2$. Then
	\begin{itemize}
		\item $r_{k,k+1}= k s_N +O(\frac{k}{N})$
		\item $r_{k,k-1}= \binom{k}{2} \frac{\rho^2}{N}.$
	\end{itemize}
	Define $r_k = r_{k,k+1}+r_{k,k-1}$ the total jump rate and
	\begin{align*}
	r_{\star}=r_{N^{1-b+\varepsilon}}=\max_{1\leq k \leq N^{1-b+\varepsilon}} r_k= N^{1-2b+2\eps} (1+o(1))
	\end{align*}
	the maximal jump rate. We aim for the coupling to hold for an interval of length $N^{b+\eps}$. The jump times of $(B_r)_{r \geq 0}$ are exponentially distributed with a parameter bounded from above by~$r_{\star}$. To estimate the number of jumps falling into an interval of length $N^{b+\eps}$ we use Theorem~5.1~iii) in \cite{SJ}. Let $(X_i)_{i \geq 1}$ be a family of independent Exp$(r_{\star})$ distributed random variables. For $c = 1-b+4\varepsilon$ Theorem 5.1 iii) yields
	\begin{align}
	\PP\left(\sum_{i=1}^{N^c} X_i \leq N^{b+\varepsilon} \right) =O(\exp( -N^{1-b+4\eps})),
	 \label{Exp bound exp variables}
	\end{align}
	that is the number of jumps is bounded by $N^{1-b+4\eps}$ with probability $1-O(\exp( -N^{1-b}))$.
	For $E=\{(B_r)_{0\leq r \leq N^b} \text{ has at most one jump in the intervals $ [j,j+1]$ for each } 0 \leq j \leq N^b-1 \}$ we have
	\begin{align*}
	\PP(E) &\geq\left( 1-O(\exp( -N^{1-b})) \right) \prod_{i=1}^{N^c}\PP(X_i >1 ) \geq \left(1-O( \exp({-N^{1-b}} )  \right)e^{-r_{\star}  N^c} \notag \\
	&= 1- O(N^{-\delta})
	\end{align*}
	which yields Claim 1.\\	
	Let $T_i^A=\inf \left\lbrace m \geq T_{i-1}^A: A_m \neq A_{T_{i-1}^A} \right\rbrace $ be the $i$-th jump of the CASP with the convention $T_{-1}^A=0$. In the same manner let $T_i^B=\inf \left\lbrace r \geq T_{i-1}^B: B_r \neq B_{T_{i-1}^B} \right\rbrace$ be the $i$-th jump of the MASP again with the convention that $T_{-1}^B=0$. We have
	\begin{align*}
	\PP(B_{T_i^B}=k+1 | B_{T_{i-1}^B} =k)  =\frac{r_{k,k+1}}{r_{k,k+1}+r_{k,k-1}}  =\PP(A_{T_i^A}=k+1| A_{T_{i-1}^A} =k) + e_{k,N}
	\end{align*}
	and
	\begin{align*}
	\PP(B_{T_i^B}=k-1 | B_{T_{i-1}^B} =k)=\frac{r_{k,k-1}}{r_{k,k+1}+r_{k,k-1}}=\PP(A_{T_i^A}=k-1| A_{T_{i-1}^A} =k)+ f_{k,N}, 
	\end{align*}
	where $e_{k,N}, f_{k,N} \in O(\max \left\lbrace  k^2 s_N^2,k^4N^{-2},N^{-1} \right\rbrace )$, the latter being the error terms from \eqref{TransitionProbCASP2} and \eqref{TransitionProbCASP3}. Note that $e_{k,N}, f_{k,N} \geq 0$ because the CASP can make jumps of size $2$ or larger. Set $d_{k,N}=e_{k,N}+f_{k,N}$.\\ 
	We show that we can couple the times $T_{i}^A$ and $T_{i}^B$, such that 
	$T_{i+1}^B <  T_i^A$ for $i=1, ..., N^{1-b + 4\eps}$ with probability 
	$1 - O(N^{-\delta})$. From that follows the Assertion \eqref{Coupling CASP ASG} of the lemma by coupling the jump directions.\\
	We couple the jump times $T_i^A$ and $T_i^B$ such that for all $i \in \{ 1,...,\lfloor N^{1-b+3\eps}\rfloor\}$
	\begin{align}
	\PP(T_{i+1}^B <  T_i^A )=O(N^{1-2b +2\eps})  \label{Coupling condition}
	\end{align}
	from which follows the assertion. We explicitly construct the coupling for $i=1$, and the same holds for any $i \in \{1,...,\lfloor N^{1-b+4\eps}\rfloor\}$. To show $\eqref{Coupling condition}$ observe that, if $A_0=k= B_0$ we can couple $T_1^A$ and $T_1^B$ by setting
	\begin{align*}
	T_1^A \stackrel{d}{=} \left\lceil  \frac{\ln U_1}{ \ln(1-r_k +d_{k,N})} \right\rceil ,\qquad T_1^B\stackrel{d}{=} -\frac{\ln U_1}{r_k}
	\end{align*}
	for $U_1 \sim $Unif$([0,1])$, since $T_1^B$ is Exp$(r_{k,k+1}+r_{k,k-1})$ distributed and $T_1^A$ is Geom$(r_{k,k+1}+r_{k,k-1}+d_{k,N})$ distributed. Note that $T_1^A \geq T_1^B$ almost surely. The coupling holds due to a) if
	\begin{align}
	\PP(T_{2}^B - T_1^B <  T_1^A-T_{1}^B )=O(N^{1-2b +2\eps}). 
	\label{Coupling condition prob}
	\end{align}
	Furthermore observe
	\begin{align*}
	T_1^A-T_1^B \geq \ln U_1 \left( \frac{1}{\ln(1-r_k+d_{k,N}) }+\frac{1}{r_k}\right)=: c_k \ln U_1
	\end{align*}
	We can upper bound the probability in \eqref{Coupling condition prob} if we assume $T_{2}^B- T_1^B \sim $Exp$(r_{k+1})$, thus we obtain for $E_2\sim$Exp$(r_{k+1})$
	\begin{align*}
	\PP( T_{2}^B - T_1^B <T_1^A -T_1^B )&\leq \pp{ E_2 \leq c_k \ln U_1 }\\&= 1- \int_0^1 e^{-r_{k+1} c_k \ln u} du \notag 
	=1 -\int_{0}^{1} u^{-r_{k+1} c_k}du\\&= 1- \frac{1}{r_{k+1} c_k+1} 
	=1-   \frac{1}{\frac{r_{k+1}}{\ln(1-r_k+O(e_N)) } +\frac{r_{k+1}}{r_k} +1} \notag \\
	&=1- \frac{1}{ -\frac{r_{k+1}}{r_k + d_{k,N}+O(r_k^2)}+\frac{r_{k+1}}{r_k} +1} \notag \\
	&=1- \frac{1}{ -\frac{r_{k+1}}{r_k} (1 + O(d_{k,N}/r_k)+O(r_k))+\frac{r_{k+1}}{r_k} +1} \notag \\
	&= O(d_{k,N}/r_k)+O(r_k) = O(N^{1-2b + 2\eps}). \notag
		\end{align*}
	which proves \eqref{Coupling condition}. Together with Claim 1 this proves the assertion of the lemma.
\end{proof}
	We are now able to complete the proof of Theorem \ref{Haldane  Formel}a.
\begin{proof}[Proof of Theorem \ref{Haldane Formel}a.]
	Let $(A_m)_{m\in \Z} = (A_m^{(N)})_{m\in \Z}$ be a stationary version of the CASP with parameters $N$, $\mathscr L(\mathscr W^{(N)})$ and $s_N$ (where as in the previous statements and proofs we are going to suppress the superscript $N$ in $A_m^{(N)}$). By Corollary \ref{corfix} it suffices to analyse $\EE{A_0}/N$ in order to obtain the probability of fixation of a single beneficial mutant. Let 
	\begin{align}\label{defE}
	\mathcal{E}:=\mathcal E^{(N)}:=\{A_{-\lfloor N^{b+\eps}\rfloor} \leq N^{1-b+\eps} ,\,  |A_{-j}-B_{-j}|\leq 1,\ \forall j \in \{0,1,...,\lfloor N^{b+\eps}\rfloor \} \}
	\end{align}
	 be the event that the (stationary) CASP $(A_m)_{m\in \Z}$ is not unusually big at time $-\lfloor   N^{b+\eps}\rfloor $ and can  be coupled with a MASP $(B_r)_{r\ge -\lfloor N^{b+\eps}\rfloor }$ for the time between  $-N^{b+\eps}$ and $0$ such that there the CASP and the MASP differ at most by $1$. Due to the Lemmata \ref{Lemma CASP comes down} and \ref{Lemma Coupling MASP,CASP} we can estimate the probability of this event by
\begin{align}
\pp{\mathcal{E}} = (1-O(N^{-\delta}))(1-O(\exp(-N^{\delta })))=1-O(N^{-\delta}) \label{Event E}
\end{align}
and a suitable  $\delta > 0$. This yields
\begin{align}
\frac{\EE{A_0}}{N}=\frac{1}{N} \EE{A_0|\mathcal{E}}\pp{\mathcal{E}}+ \frac{1}{N} \EE{A_0|\mathcal{E}^c}\pp{\mathcal{E}^c} \label{Decomposition expectation}
\end{align}
We analyse the two expectations above separately, the first one will give us the desired Haldane formula, whereas the second is an error term of order $o(s_N)$. By Lemma \ref{Lemma Moran coupling Stationarity} we get that with 
\begin{equation} \label{Beq} 
 B^{(N)}_{\rm eq}\stackrel{d}{=} {\rm Bin}(N,\frac{2s_N}{2s_N+\rho^2})
 \mbox{ conditioned to be strictly positive},
 \end{equation}
\begin{align*}
\frac{1}{N}\EE{A_0|\mathcal{E}}\pp{\mathcal{E}} & = \frac{1}{N} \sum_{j=1}^{N} j \pp{A_0=j|\mathcal{E}}(1-O(N^{-\delta})) \\ &
= \frac{1}{N} \sum_{j=1}^{N} j \pp{ B^{(N)}_{\rm eq}=j} (1-O(N^{-\delta}))\\ & = \frac{1}{N} \frac{N2s_N}{2s_N+\rho^2} (1-O(N^{-\delta}))
=\frac{2s_N}{\rho^2}(1+o(s_N)).
\end{align*}
It remains to bound the second expectation on the r.h.s.~of \eqref{Decomposition expectation}, with the worst case being $A_{-\lfloor N^{b+\eps}\rfloor}=N$. Then  using the second part of Corollary \ref{Corollary Expectation DASG} gives us
\begin{align*}
\frac{1}{N} \EE{A_0|\mathcal{E}^c}\pp{\mathcal{E}^c}= O\left( \frac{N^{1-b+\eps}}{N} N^{-\delta}\right)=O(N^{-b+\eps-\delta})=o(s_N),
\end{align*}
since $\eps>0$ can be chosen small enough such that $\delta>\eps$. This finishes the proof of Theorem~\ref{Haldane  Formel}.
\end{proof}

\begin{corollary}\label{asnorm}
	Let $A_{\rm eq}^{(N)}$ have the stationary distribution of the CASP with parameters $N$, $\mathscr L(\mathscr W^{(N)})$ and $s_N$. Then, with  $p_N:=\frac{s_N }{\rho^2/2+s_N }$, $\mu_N:= Np_N$  and $\sigma_N^2=Np_N(1-p_N)$, the sequence of random variables
	$\left(A_{\rm eq}^{(N)}-\mu_N\right)/\sigma_N$ 
 converges as $N\to \infty$  in distribution to a standard normal random variable.
	\end{corollary}
\begin{proof}
	In the previous proof we have worked, for a stationary CASP $(A_m^{(N)})_{m\in \mathbb Z}$, with a decomposition of $\mathbb E[A_{0}^{(N)}]$ according to the events  $\mathcal{E}$ and $\mathcal{E}^c$, with $\mathcal{E}$ defined in  \eqref{defE}. We now use the same decomposition for the distribution of  $(A_0^{(N)}-\mu_N)/\sigma_N$
	and obtain for any  $f \in C_b(\R)$ with the same line of reasoning as in the previous proof and with $B^{(N)}_{\rm eq}$ as in \eqref{Beq}
		\begin{equation*}
	\lim_{N\to \infty} \mathbb E[f((A_0^{(N)}-\mu_N)/\sigma_N)] = \lim_{N\to \infty} \mathbb E[f((B^{(N)}_{\rm eq}-\mu_N)/\sigma_N)]= \mathbb E[f(Z)],
	\end{equation*}
where $Z$ is a standard normal random variable.
\end{proof}
\begin{remark}
Using the technique of \cite{K1} it is not difficult to show that $\left(A_{\lfloor r{s_N^{-1}}\rfloor}^{(N)}/\mu_N\right)_{r \ge 0}$ converges in distribution as $N\to \infty$  uniformly on compact time intervals to the solution of a dynamical system whose stable fixed point is 1. One might then also ask about  the asymptotic fluctuations of the   process $A^{(N)}$. {\color{black} Although available results in the  literature (like \cite[Theorem 8.2]{K2} or \cite[Theorem 11.3.2]{EK}) do not directly cover our situation (because e.g. of boundedness assumptions required there), the coupling between $A^{(N)}$ and $B^{(N)}$ analysed above is a promising tool to obtain weak convergence of properly rescaled ancestral processes $A^{(N)}$ to an Ornstein-Uhlenbeck process, which in view of Corollary \ref{asnorm} should include also time infinity.}
Let us mention in this context \cite{C}, which contains a fluctuation result (including time infinity) for the Moran frequency process under \emph{strong} selection and two-way mutation.
\end{remark}

\section{A concentration result for the  equilibrium distribution of the CASP.  Proof of Theorem \ref{Haldane  Formel}b.}\label{Sec7}
Let  $A^{(N)}= (A_m^{(N)})_{m\geq 0}$ be the Cannings ancestral selection process (CASP) as defined in Sec.~\ref{secCASP}.
We will show in the present section that under the assumptions \eqref{sasymptoticsa}, \eqref{second} and \eqref{moments}   the expectation of the equilibrium state $A_{\rm eq}^{(N)}$ of $A^{(N)}$ satisfies the asymptotics
\begin{equation}\label{toshow}
   	\EE{A_{\rm eq}^{(N)}}= \frac{2}{\rho^2 }  s_N N (1+o(1)).
   	\end{equation}
 The proof of Theorem  \ref{Haldane  Formel}b is then immediate from Corollary \ref{corfix}. 
  
  Let us describe here the strategy of our proof. We will show that the distribution of $A_{\rm eq}^{(N)}$ is sufficiently concentrated around the ``center''  $\frac{2}{\rho^2 }  s_N N$ as  $N\to \infty$.
  Throughout, we will fix a sequence $(h_N)$ obeying \eqref {ellN} such that  \eqref{moments} is satisfied.  As in the previous section we will switch to $b_N$ defined by \eqref{stob}. The Assumption  \eqref{sasymptoticsa}, which is now the standing one, thus translates into 
  \[\frac 12 +\eta \le b_N \le 1-\eta.\]
  Frequently we will suppress the subscript $N$ in $b_N$, thus denoting the sequence $s_N N$ simply by~$N^{1-b}$. We will show in the subsequent lemmata that the CASP $A^{(N)}$ needs only a relatively short time to enter a small box around  $\frac{2}{\rho^2 }  s_N N$, compared to the time it spends in this box.  The former assertion is provided by Lemmata \ref{Lemma Coming down} and \ref{Lemma going up from 1}. The behaviour of $A^{(N)}$ near the center is controlled by Proposition \ref{Lemma exp leaving time}. This is prepared by Lemmata \ref{Lemma Probability large upward jump} and~\ref{Lemma Log N Down} which bound the probability of jumps of absolute size larger than $h_N$ near the center. {\color{black} The estimates achieved in the lemmata allow to bound from above and below the process $\mathcal A:= A^{(N)}$ by processes $\mathcal A^u$ and $\mathcal A^{\ell}$ on an event of high probability. The process $\mathcal{A}^u$ moves only in the box $\mathcal{I}^u = [n^{(\gamma)}, n^{(\alpha)}]$ which is close to the center (i.e. $n^{(\gamma)} - \frac{2}{\rho^2} N^{1-b}$ and $n^{(\alpha)}  - \frac{2}{\rho^2} N^{1-b} \ll N^{1-b}$). All (upward or downward) jumps of $\mathcal A$ of size $2, \ldots, h_N$ are replaced in $\mathcal{A}^u$ by an upward jump of size $h_N$, furthermore $\mathcal{A}^u$ is reset to its starting value $n^{(\beta)}$ near the lower boundary of the box $\mathcal{I}^u$, see also Figures \ref{Aupper} and \ref{ProofLeaveCenter} for illustrations. The precise definitions of $\mathcal A^u$ and $\mathcal A^{\ell}$ are given in the proof of Proposition \ref{Lemma exp leaving time}.}
  	\begin{figure}[ht]
		\centering
		\includegraphics[width=14cm]{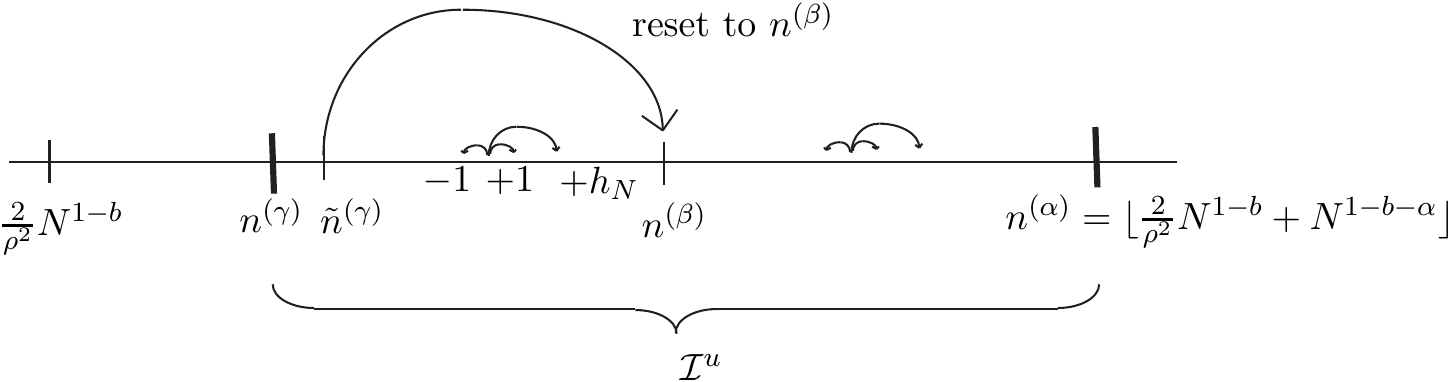}
			\caption {{\color{black}A sketch of the transition dynamics of the process $ \mathcal A^u$ which upper bounds stochastically the CASP with high probability. Mainly $ \mathcal A^u$ makes jumps only of size $\pm 1$, occasionally it jumps upwards by $h_N$ and whenever it reaches $\tilde n^{(\gamma)}$, it is reset to its starting point $n^{(\beta)}$.  The precise definitions of the quantities $n^{(\beta)}$ and $\tilde n^{(\gamma)}$ as well as of the process ${\mathcal A}^u$ are given in the proof of Proposition \ref{Lemma exp leaving time}.}
			}
			\label{Aupper}	
			\end{figure}

The following lemma controls the probability of large upward jumps of $\mathcal A$ near the center, using the construction of the branching step of the CASP described in Section \ref {secCASP}.
\begin{lemma}[Probability for large jumps upwards]\label{Lemma Probability large upward jump} \qquad \\
	Let $k = \lceil \kappa N^{1-b} \rceil$ for some $\kappa > 0$, then
	\begin{align}
	\pp{  A_{m+1} \geq A_m + h_N| A_m=k } = O( (N^{1-2b} )^{h_N}) \mbox { as } N\to \infty.
	\end{align}
	\end{lemma}
\begin{proof}
	We want to estimate 
	\begin{equation}
	\pp{  A_{m+1} \geq A_m + h_N| A_m=k } \leq  \pp{  \sum_{i=1}^{k} G^{(i)} \geq k+ h_N }, 
	\end{equation}
	for independent Geom$(p)$ with $p=1-s_N$ distributed random variables $G^{(i)}$, $i \geq 1$.
	With  $k'= k +h_N$, $S_{k'}$ a  Bin$(k',p)$-distributed random variable, and $a= \frac{k}{k'}$, we can estimate the r.h.s.~from above by
	\begin{equation}
	 \pp{S_{k'} \leq k} = \pp{S_{k'} \leq ak'}, 
	\end{equation}
since the probability that at least $k'$ trials are necessary for $k$ successes can be estimated from above by the probability to have at most $k$ successes in $k'$ trials. Using the Chernoff bound for binomials we can estimate 
		\begin{align}\label{Chernoff Binomial}
	\pp{{S_{k'}} \le ak'} \leq \exp( - k'  I(a)), 
		\end{align}
	with rate function $I(a)=a \ln \left( \frac{a}{p}\right) + (1-a) \ln \left( \frac{1-a}{1-p}\right)$.
	Inserting our parameters yields
	\begin{align}
	I(a)&=\frac{k}{k + h_N} \ln \left( \frac{ k}{k + h_N} \frac{1}{1-s_N}\right)+\frac{h_N}{h_N + k} \ln \left( \frac{h_N}{k + h_N} \frac{1}{s_N}\right)   \\
	&= \frac{\kappa N^{1-b}}{\kappa N^{1-b} + o(N^{1-b})} \ln \left( 1- \frac{h_N}{\kappa N^{1-b}}(1+o(1))\right)+\frac{h_N}{\kappa N^{1-b} +o(N^{1-b})}  \ln \left( \frac{ N^b h_N }{\kappa N^{1-b}+o(N^{1-b})} \right) \\
	&=  \frac{h_N}{\kappa N^{1-b}} \ln \left( \frac 1\kappa N^{2b-1}  h_N \right) (1+o(1)) - \frac{1}{\kappa} N^{b-1}h_N (1+o(1)).
	\end{align}
	The dominating term above is $\frac{1}{\kappa} N^{b-1} h_N \ln(N^{2b-1}) $ and plugging this back into \eqref{Chernoff Binomial} one obtains
	\begin{align}
	\pp{S_{n'} \leq an'} &\leq \exp( - \kappa N^{1-b} \frac{1}{\kappa} N^{b-1} h_N \ln(N^{2b-1}) ( 1+o(1) ) ) \\
	&=\exp(- h_N \ln (N^{2b-1})) (1+o(1) ) =  (N^{1-2b})^{h_N} (1+o(1)).
	\end{align}
	
\end{proof}

Next, we set out to bound the probability for downward jumps of size at least $h_N$ near the center. In view of the construction of the coalescence step described in Section \ref {secCASP} this is settled by the following lemma. 

\begin{lemma}[Probability for large jumps downwards] \label{Lemma Log N Down}\qquad  \\
Let $\mathscr{W}^{(N)}=(W_1^{(N)},...,W_N^{(N)})$ be as in Theorem 3.5b, let $\kappa>0$ and $k^{(N)} \leq \kappa N^{1-b}$. 
For $N\in \mathbb N$ sort \mbox{$k^{(N)}$} balls independently into $N$ boxes, such that any given ball is sorted into box $i$ with probability  $W_i^{(N)}$, $i\in [N]$. Then the probability that no more than $k^{(N)}-h_N$ boxes are occupied  is
\begin{equation}
O\left( h_N^4 N^{1-2b} \right)^{h_N} \mbox{ as } N\to \infty. 
\end{equation}
\end{lemma}
\begin{proof}
   We will suppress the superscript $(N)$ and write $k:= k^{(N)}$, $\mathscr W := \mathscr{W}^{(N)}$. For $h:= h_N$ let~$p_h$ be the probability of the event that no more than $k-h$ boxes are occupied.  
    This is equal to the event that at least $h$ collisions occur, where we think of the balls with numbers $1,\ldots,k$ being subsequently sorted into the boxes  and say that the ball with number $\nu$ produces a collision if it lands in an already occupied box.    In the following we record the occupation numbers of (only) those boxes that receive more than one ball. These are of the form $\beta=(\beta_1, ...,\beta_\ell) \in \{2, ...,h+1\}^\ell$ with  $\ell \in \{1, ..., h\}$ and $\beta_1+...+\beta_\ell -\ell= h$. For a given $\beta$ of this form, and $\ell$ given boxes with  $|\beta|:= \beta_1+ ...+ \beta_{\ell}$, assume that $\beta_1$ balls are sorted into the first box, $\beta_2$ balls into the second box, etc, and the remaining $k- |\beta|$ balls are sorted into arbitrary boxes (so that, as required, the number of occupied boxes is at most $\ell + k - |\beta| = k- h$). 
   Given the weights $W_1, \dots , W_N$, the probability to sort the first $\beta_1$ balls into box 1, the following $\beta_2$ balls into box $2,\dots ,$ and finally $\beta_\ell$ balls into box $\ell$ is
    \begin{equation*}
    \prod_{i=1}^{\ell} W_i^{\beta_i}.
    \end{equation*}
   
   There are $\frac{N!}{(N-\ell)!}$ many possibilities to choose $\ell$ different boxes out of $N$. Furthermore, there 
   are $\binom{k}{\beta_1,\ldots,\beta_\ell,k-|\beta|}$ many possibilities to choose $|\beta|$ many balls out of $k$ balls 
   and sort these balls into $\ell$ boxes, such that $\beta_i$ balls are sorted into box $i$. Hence, due to exchangeability of the weights $W_1, \dots , W_N$ we get
    \begin{equation}\label{JumpAtLeastJ}
	p_h\leq  \sum_{\beta\in \mathcal B} \EE{\prod_{i=1}^{\ell(\beta)} W_i^{\beta_i}}  \frac{N!}{(N-\ell(\beta))!} \binom{k}{\beta_1,\ldots,\beta_\ell,k-|\beta|},
	\end{equation}
   where 
    \begin{equation} \mathcal B:= \bigcup_{\ell\in\{1,\ldots, h\}} \{\beta=(\beta_1,\ldots, \beta_\ell): \beta_1+...+\beta_\ell -\ell= h \mbox{ and } \beta_1 \geq \beta_2 \geq ...\geq \beta_\ell\}
\end{equation}
 and $\ell = \ell(\beta)$ denotes the length of the vector $\beta \in \mathcal B$.
 
  To obtain an upper bound of the r.h.s.~of  \eqref{JumpAtLeastJ} we estimate the moments $\EE{ \prod _{i=1}^{\ell} W_i^{\beta_{i}}}$. 
   Since $(W_1,W_2,...,W_N)$ are negatively associated \cite{Joag-Dev1983}, we can use the property 2 in \cite{Joag-Dev1983} of negatively associated random variables, which reads   
 \begin{align}
	\EE{ \prod _{i=1}^{\ell} W_i^{\beta_{i}}} 
	\leq  \prod _{i=1}^{\ell} \EE{ W_i^{\beta_i}}.
	\end{align}
   Applying Jensen's inequality we can estimate for $\beta_1 \geq \beta_2$
	\begin{eqnarray*}
	\EE{W_1^{\beta_1 }} \EE{W_2^{\beta_2}} &\leq& \EE{W_1^{\beta_1 }} \EE{W_2^{\beta_1}}^{\frac{\beta_2}{\beta_1}}= \EE{W_1^{\beta_1}}^{1+\frac{\beta_2}{\beta_1}} \\
	&\leq& \EE{W_1^{\beta_1+\beta_2}}.
	\end{eqnarray*}
Iterating the above argument, we obtain
	\begin{equation}\label{Iteration Jensen}
	 \prod _{i=1}^{\ell} \EE{ W_i^{\beta_i}}  \leq 
	 \EE{W_1^{|\beta|}}.
	\end{equation}
	\blue{With regard to \eqref {JumpAtLeastJ} and  \eqref {Iteration Jensen} we will now analyse the quantities
	\begin{align}
	a_\beta:= \EE{W_1^{|\beta|}} \frac{N!}{(N-\ell(\beta))!},\qquad \beta\in \mathcal B. \label{probjump2}
	\end{align}
	For brevity we write $\ell(\beta)=:\ell$. Since $|\beta| = h + \ell$ and $1\le \ell \le h$, we obtain from  \eqref{moments} for $N$ sufficiently large and all $\beta \in \mathcal B$ the estimate
\begin{equation}\label{aest}
a_\beta \le  \frac{(Kh)^{h+\ell}}{N^hN^\ell} N(N-1)\cdots (N-\ell+1) \le   \frac{(Kh)^{2h}}{N^h}.
\end{equation}}
	For the rightmost term in \eqref{JumpAtLeastJ} we have the estimate 
	\begin{equation}\label{multcoeff}
	\binom{k}{\beta_1,\ldots,\beta_\ell,k-|\beta|} \le  \binom{k}{2,...,2,k-2h_N}\le (\kappa N^{1-b})^{2h_N} 
	\end{equation} 
	 and the number of occupation vectors $\beta$ appearing in the sum in \eqref{JumpAtLeastJ} (i.e. the cardinality of~$\mathcal B$) can be estimated from above by  $(h_N +1)^{  h_N}$.
	Hence we obtain from  \eqref{JumpAtLeastJ}, \eqref{aest} and \eqref{multcoeff},  
	{\color{black}
	\begin{align*}
	 p_{h_N} &= O\left((\kappa N^{1-b})^{2h_N}N^{-h_N} (K h_N)^{2h_N}(h_N+1)^{ h_N}\right) \\	
	 &= O \left( (N^{1-2b}h_N^4 )^{h_N}\right).\nonumber
	\end{align*}
	}
\end{proof}

Building on the previous two lemmata, the next result shows that the CASP does not leave the central region up to any polynomially long time{\color{black}, i.e. within a time frame of order $N^c$ for any $c>0,$} with high probability.

\begin{proposition}[CASP stays near the center for a long time] \label{Lemma exp leaving time}\qquad \\
	Consider $\alpha$, $\beta$ with  $ 0<\alpha<\beta < \frac{2b-1}{3}$, let
	\[\mathcal{I}:=\left[\left\lceil \frac{2}{\rho^2 } N^{1-b} -  N^{1-b-\alpha} \right\rceil , \left\lfloor \frac{2}{\rho^2 } N^{1-b} +  N^{1-b-\alpha} \right\rfloor \right]\] and define $\tau_{\mathcal{I}}:=\inf \{ m \geq 0 : A_m \notin \mathcal{I} \}$. Then for all $\theta>0$ and all $\eps>0$
	\begin{equation}
	\pp{\tau_{\mathcal{I}} \leq N^\theta | A_0=k} = O((N^{1-2b+\eps})^{h_N})   \quad \mbox{ as } N\to \infty \label{ProbLeave}
	\end{equation}
	uniformly in 
	$k \in  [ \lceil \frac{2}{\rho^2 } N^{1-b} - N^{1-b-\beta} \rceil , \lfloor \frac{2}{\rho^2 } N^{1-b} + N^{1-b-\beta} \rfloor ]$.
\end{proposition}

	\begin{figure}[ht]
		\includegraphics[width=13cm]{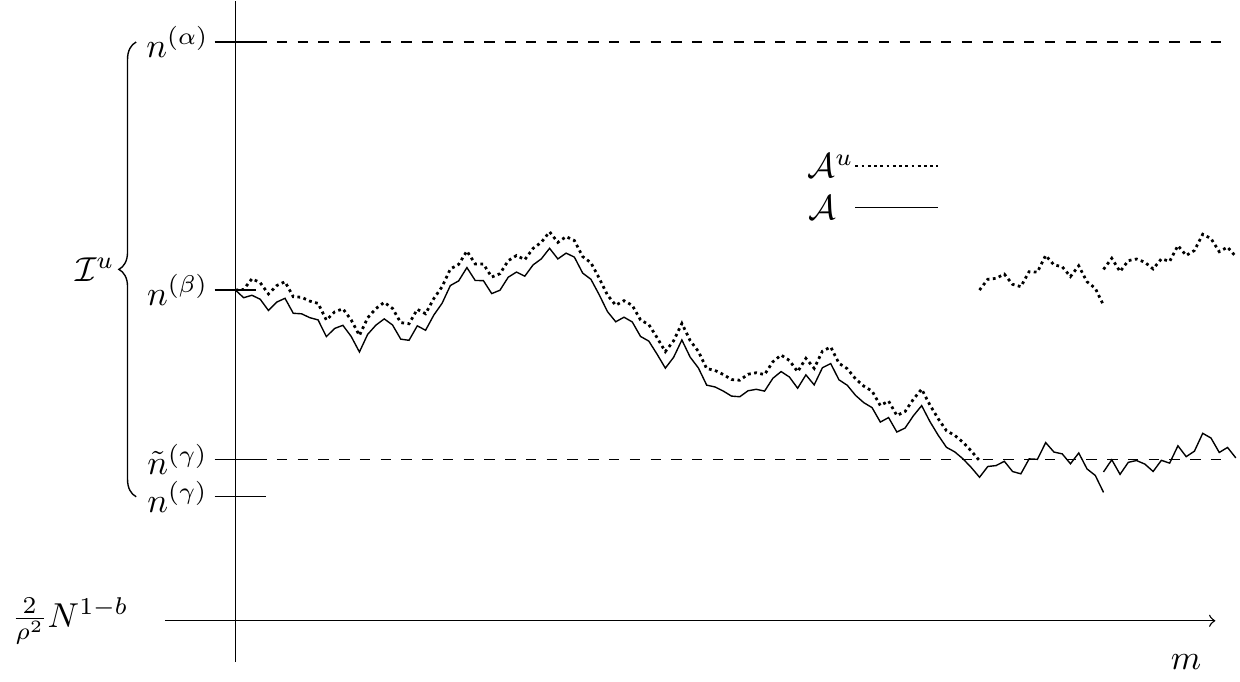}
			\caption {
			 An example of a realisation of the processes $\mathcal{A}$ and $\mathcal{A}^u$, displaying that  $\mathcal{A}^u$ dominates~$\mathcal{A}$ as long as it is below the level $n^{(\alpha)}$. Note that $\mathcal{A}^u$ is reset to $n^{(\beta)}$ whenever it hits the level $\tilde{n}^{(\gamma)}$; see also Figure \ref{Aupper}.
			}
			\label{ProofLeaveCenter}	
			\end{figure}

\begin{proof}
	To show the above claim we bound stochastically $\mathcal{A}$ from above and from below by simpler processes, which in certain boxes close to the center of attraction of $\mathcal{A}$ follow essentially a time changed random walk dynamics with constant drift. In the first part of the proof we will construct a time-changed Markov chain $\mathcal A^u$ that dominates $\mathcal A = (A_m)_{m \geq 0}$ from above for a sufficiently long time. This construction will rely on Lemma \ref{Lemma Probability large upward jump}.  In this first part we will give all details; in the second part of the proof we will indicate how an analogous construction can be carried out ``from below'', then making use of \ref{Lemma Log N Down}.
	\\
	1. Let $\gamma \in (\beta,  \frac{2b-1}{3})$  and write, in  accordance with Figure \ref{Aupper}, 
\[n^{(\zeta)} :=  \lceil  \tfrac{2N^{1-b}}{\rho^2} + N^{1-b -\zeta} \rceil, \qquad \zeta \in\{\alpha , \beta, \gamma \}.\]
Observing that $n^{(\gamma)} <  n^{(\beta)} < n^{(\alpha) }$, we  consider the box
\[\mathcal{I}^u:= \left[n^{(\gamma)}, n^{(\alpha) }\right],\]
and take $n^{(\beta)}$ as the starting point of both $\mathcal A$ and $\mathcal A^u$. The process $\mathcal{A}^u$ will be such that $\mathcal{A}^u$ makes only jumps of size $-1,1,h_N$ and it is re-set to its starting point $n^{(\beta)}$ as soon as it hits the level
\[\tilde n^{(\gamma)} := n^{(\gamma)} + h_N.\]
Consider the following Markov chain ${\mathcal{ \bar{A}}}^u$.
We decree that within the box $(\tilde n^{(\gamma)},n^{(\alpha) }]$ the process
${\mathcal{ \bar{A}}}^u$ makes only jumps $-1,+1$ and $+h_N$. Here, the probabilities for jumps
$-1,+1$ of ${\mathcal{ \bar{A}}}^u$ from an arbitrary state in $(\tilde n^{(\gamma)},n^{(\alpha) }]$ 
are set equal to the  probabilities  for jumps $-1,+1$ of $\mathcal A$ from the state $ n^{(\gamma)}$, and
the probability for a jump $+h_N$  of ${\mathcal{ \bar{A}}}^u$ from an arbitrary state in $(\tilde n^{(\gamma)},n^{(\alpha) }]$ is set equal to the  probability  of a jump of $\mathcal A$ from the state $ n^{(\gamma)}$ that has an absolute size larger than 1. 
More formally, we define
\begin{equation}
c_N^{(\gamma)}= \mathbbm{P}\big[|A_{m+1} -A_{m}|>0 | A_m= n^{(\gamma)}\big]
\end{equation}
and observe that
\begin{equation} 
c_N^{(\gamma)}= \mathbbm{P}\big[|A_{m+1} -A_{m}|>0 | A_m= n^{(\gamma)}\big]= s_N n^{(\gamma)} + \binom{n^{(\gamma)}}{2} \frac{\rho^2}{N}(1 + o(1)).
\end{equation}

We define the jump probabilities as follows:
	\begin{align*}
	\mathbbm{P}(\bar{A}^u_{i} = k +1 \mid  \bar{A}^u_{i -1} =k)& = \frac 1{ c_N^{(\gamma)} } s_N n^{(\gamma)}, \qquad k> \tilde n^{(\gamma)},
	 \\
	\mathbbm{P}(\bar{A}^u_{i} =k+ h_N \mid \bar{A}^u_{i -1} =k) & = 1 -  \frac {1}{c_N^{(\gamma)}} \left(s_N n^{(\gamma)} + \binom{n^{(\gamma)}}{2} \frac{\rho^2}{N}\right), \qquad k> \tilde n^{(\gamma)}\\
	\mathbbm{P}(\bar{A}^u_{i} = k -1 \mid \bar{A}^u_{i -1} =k)& =   \frac {1}{c_N^{(\gamma)}} \binom{n^{(\gamma)}}{2} \frac{\rho^2}{N}, \qquad k > \tilde n^{(\gamma)},\\
	\mathbbm{P}(\bar{A}^u_{i} =n^{(\beta)}| \bar{A}^u_{i -1} = \tilde{n}^{(\gamma)})& =1.
	\end{align*}
Note that	
\begin{align}
		\begin{split}
	&\frac {1}{c_N^{(\gamma)}} s_N n^{(\gamma)}= \frac{1}{2}\left(1 - \frac{\rho^2}{4} N^{-\gamma}\right)(1+O(s_N n^{(\gamma)} )),
	\label{boxdrift} \\
	&\frac {1}{c_N^{(\gamma)}} \binom{n^{(\gamma)}}{2} \frac{\rho^2}{N}= \frac{1}{2}\left(1 + \frac{\rho^2}{4} N^{-\gamma}\right)(1+O(s_N n^{(\gamma)} )),\\
	&1 -  \frac {1}{c_N^{(\gamma)}} \left(s_N n^{(\gamma)} + \binom{n^{(\gamma)}}{2} \frac{\rho^2}{N}\right) = O(N^{1-2b}), 
	\end{split}
\end{align}
	which results in a small downwards drift of $\mathcal{\bar{A}}^u$ in $\mathcal{I}^u$.
	The process $\mathcal{A}^u$ is defined as follows. Denote by $\tau_i$ the time of the $i$-th non-trivial
	jump (that is a jump of size $\neq 0$) of $\mathcal{A}$ for $i\geq 1$ and let $\tau_0=0$. We set for 
	$\tau_{i-1}\leq m \leq \tau_{i} -1$, with $i\geq 1$, 
	\begin{align*}
	 A^u_m := \bar{A}^u_{i-1}.
	\end{align*}
    
	Putting $T_N:= \lceil N^\theta\rceil$, we now consider the events
	\[E_N:= \{\mathcal A^u \mbox{ has not left  } \mathcal{I}^u \mbox{ by time  } T_N\}\]
	and 
	\[F_N:= \{\mathcal A \mbox{ has not performed jumps of absolute size larger than  }h_N \mbox{ by time  } T_N\}.\]
	We can now couple  $\mathcal A$ and $\mathcal A^u$ such that on the event $E_N \cap F_N$ and for all $m\le T_N$ we have $A_m \le A_m^u$. In order to show that the probability of the event $\{\mathcal A$ reaches $n^{(\alpha)}$ before time $T_N\}$ is bounded by the r.h.s.~of \eqref{ProbLeave} it thus suffices to show that  
	\begin{equation}\label{both}
	\mathbb P(E_N^c)= O((N^{1-2b +\eps})^{h_N}) \quad \mbox{ and } \quad \mathbb P(F_N^c)= O((N^{1-2b +\eps})^{h_N})  \quad\mbox{as } N\to \infty.
	\end{equation}
 From Lemmata \ref{Lemma Probability large upward jump} and  \ref{Lemma Log N Down} it is obvious that
	\begin{equation}\label{PF}
	\mathbb P(F_N^c) =O(T_N (h_N^4 N^{1-2b})^{h_N} ).
	\end{equation} 	
We claim that there exists a $\delta >0$ such that
	\begin{equation}\label{PE}
	\mathbb P(E_N^c) = O(\exp(-N^{\delta})).
	\end{equation}
Write $p_{\rm hit}$ for the probability that $\mathcal A ^u$, when started from $n^{(\beta)}$, hits (or crosses) $n^{(\alpha)}$ before it hits $\tilde n^{(\gamma)}$.
The jump size of the process $\mathcal{A}^u$ is at least -1 at each generation. Therefore, at least $N^{1-b-\beta}$ generations are necessary to reach  the level $n^{(\gamma)}$ when starting from level $n^{(\beta)}$. Thus, with the $\theta$ given in the proposition, within $N^\theta$ generations the process $\mathcal{A}^u$ makes at most $N^\theta/N^{1-b-\beta}$ transitions from $n^{(\beta)}$ to $n^{(\gamma)}.$\\ 	
The probability that $\mathcal{A}^u$ crosses $n^{(\alpha)}$ within a single excursion from $n^{(\beta)}$ that reaches $\tilde n^{(\gamma)}$ is obviously bounded from above by the probability that the process $\widetilde{A}_m:=\sum_{i=1}^{m} X_i $, $m=0,1,\ldots$, crosses the level $n^{(\alpha)}$, where the $X_i$ are i.i.d. and distributed as the jump sizes of the process $\mathcal{\bar{A}}^u$ in a state $x > \tilde{n}^{\gamma}$; thus, other than the process $\mathcal{\bar{A}}^u$ the process $\widetilde{A}$ is not reset to $n^{(\beta)}$ at an attempt to cross  $\tilde n^{(\gamma)}$. 
For $\lambda_N = N^{-2\gamma}$ one has 
\begin{align*}
\EE{\exp(\lambda_N X_1)} = 1- \frac{\rho^2}{4} N^{-3\gamma} + O(N^{1-2b}) + o(N^{-3\gamma}) <1,
\end{align*}
for $N$ large enough, since $\gamma < \frac{2b-1}{3}$. Therefore the process $Y_m = \exp \left( \lambda_N \sum_{i=1}^{m} X_i\right)$ is a supermartingale. Define $\tilde{\tau} = \inf \{ m \geq 0 : \widetilde{A}_m \geq n^{(\alpha)}  \}$, then by the Martingale Stopping Theorem
\begin{align}
n^{(\beta)} = \EE{Y_0} \geq \EE{Y_{m \wedge \tilde{\tau }}} \geq \EE{Y_{\tilde{\tau }} \1_{ \{  \tilde{\tau} <m \} } } \geq \exp( \lambda_N n^{(\alpha)}) \PP(\tilde{\tau} <m),
\end{align}
which yields the bound  $p_{\rm hit} \leq n^{(\beta)} \exp(-\lambda_N n^{(\alpha)} )$. Thus, we can estimate
\begin{align}
\PP(E_N) \geq (1-p_{\rm hit})^{N^{\theta -1+b+\beta}},
\end{align}
which proves \eqref{PE}. Clearly, \eqref{PE} and \eqref{PF} imply \eqref{both}, which completes the first part of the proof.\\\\
	2. It remains to prove also the ``lower part''  of  \eqref{ProbLeave}, i.e. to control the time it takes  $\mathcal{A}$ to leave $\mathcal{I}$ in downwards direction. We argue similarly by defining a process $\mathcal{A}^\ell$ which bounds $\mathcal{A}$ from below in a box $\mathcal{I}^{\ell}$ correspondingly chosen to the box $\mathcal{I}^u$. This process is again a Markov chain that makes   jumps of size 1 and -1 and (rarely) jumps of size $-h_N$ and whose drift coincides with that of $\mathcal{A}$ at the upper boundary of the box $\mathcal{I}^\ell$.  Due to Lemma~\ref{Lemma Log N Down} jumps downwards at least of size  $h_N$ occur with exponentially small probability and hence, these jumps can be ignored in the time frame of interest and $\mathcal{A}^\ell$ is stochastically dominated by $\mathcal{A}$ with sufficiently high probability.
\end{proof}
We now show that the time to reach the center from state $N$ does not grow faster than polynomially in $N$, i.e. is of the order $N^{c}$ for some $c>0$ with high probability.

\begin{lemma}[Coming down from $N$] \label{Lemma Coming down} \qquad \\
	Let $0< \varepsilon <\frac{2b-1}{3}$, $B:=[\frac{2}{\rho^2 }  N^{1-b} - 2 N^{1-b-\eps },\frac{2}{\rho^2 }  N^{1-b} + 2 N^{1-b-\eps }]$ and $\tau_B := \inf \{m \geq 0 : A_m \in B\}$. Then for  any $\eps < \eps'< \frac{2b-1}{3}$ and  $0<\delta\leq \eps' -\eps$
	\begin{equation}
	\pp{ \tau_B \geq N^{2b+2\eps'}\mid A_0=N} = O(\exp( - N^{\delta})).
	\end{equation}
\end{lemma}

\begin{proof} The proof will be divided in three steps.\\ 
	1. The coalescence probabilities of the Wright-Fisher model are the smallest in our class of Cannings models with selection, see Lemma \ref{WFext}.  Therefore, the stopping time  $\tau_B$ for the CASP is stochastically dominated from above by the corresponding stopping time in a  Wright-Fisher model with selection. Consequently, we assume in the following that $\mathscr{W} = (1/N, ..., 1/N)$ and thus $\rho^2=1$.\\
2. We analyse the drift of  $(A_m)_{m \geq 0}$ in each point $y$ in $B'=[2 N^{1-b} + 2 N^{1-b-\eps'} ,N]$. We will show in part 3 of the proof that
	\begin{align}\label{DriftEst}
	y-\EE{A_1|A_0=y} \geq  2N^{1-2b-\eps'}, \qquad \text{for all } y \in B'.
	\end{align}
	The estimate \eqref{DriftEst} on the drift of $\mathcal{A}$ in $B'$ yields that for $m_0  = \lceil N^{2b+\eps'} \rceil $ it holds $\EE{A_{m_0}} \leq  2 N^{1-b} + 2 N^{1-b-\eps}$, since by Proposition \ref{Lemma exp leaving time} after entering the box $B$ the process $\mathcal{A}$ does not leave the box up to any polynomially long time except on an event with probability $O((N^{1-2b+\eps})^{h_N})$. Applying Markov's inequality yields for $m_0= \lceil N^{2b+\eps'} \rceil $
	\begin{align}
	\PP( A_{m_0} \geq 2 N^{1-b} + 2 N^{1-b-\eps}) &\leq \frac{\EE{A_{m_0}}}{2 N^{1-b} + 2 N^{1-b-\eps}} \leq \frac{2 N^{1-b} +2 N^{1-b-\eps'}}{2 N^{1-b} + 2 N^{1-b-\eps }} \\
	&=1+ \frac{N^{1-b-\eps'} -N^{1-b-\eps}}{N^{1-b} +N^{1-b-\eps}} = 1 + \frac{N^{- \eps'} - N^{-\eps}}{ 1+ N^{-\eps}} \\
	&=1- N^{-\eps}( 1+ N^{-(\eps' - \eps)} )( 1+ O(N^{-\eps})). 
	\end{align}
	Hence, after time $\lceil N^{2b+\eps'} \rceil $ the process started in $N$ is with a probability of at least $1- N^{-\eps} (1+o(1))$ in $B$. If the process $\mathcal{A}$ did not enter $B$ until time $\lceil N^{2b+\eps'} \rceil $, in the worst case the process is still in the state $N$. Therefore, recalling that $\delta \le \eps'-\eps$,  the probability that the process is after time $\lceil N^{2b+2 \eps'} \rceil $ still above $B$ can be estimated from above via
	\begin{eqnarray}
	\PP( A_{\lceil N^{2b+2\eps'} \rceil   } \geq 2 N^{1-b} + 2 N^{1-b-\eps}) &\leq& \PP( A_{m_0} \geq 2 N^{1-b} + 2 N^{1-b-\eps})^{N^{\eps'}} \leq (1-N^{-\eps})^{N^{\eps' }} \\
	&=& O(\exp( - N^{\delta})).
	\end{eqnarray}
	\noindent 
	3. It remains to show \eqref{DriftEst}. Recalling the ``balls in boxes'' description of the one-step transition probability of the CASP as described in Section \ref{secCASP},  let $\1_{C_i}$ be the indicator of the event $C_i$ that the $i$-th box is occupied by at least one ball. We can rewrite
	\begin{align}
	y-\EE{A_1|A_0=y} &= y- \EE{\sum_{i=1}^{N} \1_{C_i}|A_0=y}\\
	&=y- N\left( 1-\EE{\left(1-\frac{1}{N} \right)^H} \right) \label{Generating function H}
	\end{align}
	with $H=\sum_{i=1}^{y} G^{(i)}$ with $G^{(1)},...,G^{(y)}$ independent and $G^{(i)} \sim \text{Geom}(1-s_N)$. The expectation in \eqref{Generating function H} is the generating function of a negative binomial distribution with parameters $y$ and $(1-s_N)$ evaluated at $1-\frac{1}{N}$, which allows to continue  \eqref{Generating function H} as
	\begin{align} y- N \left(1- \left[ \frac{(1-s_N)(1-\frac{1}{N})}{1-s_N(1-\frac{1}{N})}\right]^y \right)=y-N \left( 1- \left[1- \frac{1}{N} \frac{1}{1-s_N  +\frac{s_N}{N}} \right]^y \right). \label{Estimate from below}
	\end{align}
	Using Taylor expansion for the function $x\mapsto (1-x)^y$ around $0$ with the remainder in Lagrange form yields
		\begin{equation}\label{taylor1}
	(1-x)^y = 1-yx + \binom{y}{2} x^2 - \binom{y}{3} x^3 (1-\xi)^{y-3} \geq 1-yx + \binom{y}{2} x^2 - \binom{y}{3} x^3
	\end{equation}
	for some $\xi \in [0,x]$. Abbreviating  $u:= \frac{1}{1-s_N  +\frac{s_N}{N}}$ in \eqref{Generating function H} and using \eqref{taylor1} yields that the r.h.s of \eqref{Estimate from below} can be estimated from below by
	\begin{align}
	 &\geq y-N \left( 1- \left[1- y \frac{u}{N} + \binom{u}{2} \frac{u^2}{N^2}- \binom{y}{3} \frac{u^3}{N^3} \right] \right) \\
	&= y- u y + \frac{u^2}{N} \binom{y}{2} - \frac{u^3}{N^2} \binom{y}{3} =:  h(y).
	\end{align}
	In order to show that the polynomial $h(y)$ is bounded away from $0$ as claimed in \eqref{DriftEst} we check that $h$ is
	positive at the lower boundary and that the derivative $h'$ is positive on the interval $B'$.
	We can factorise $h$ as $h(y) = y\tilde h(y)$, with
	\begin{align}
	\tilde h(y)=  1-u + \frac{u^2}{2 N } (y-1) - \frac{u^3}{6N^2} (y-1)(y-2).
	\end{align}
		It is straightforward to check for $y_0=2 N^{1-b} + 2N^{1-b-\eps'}$ that $\widetilde{h}(y_0) >0$. Thus it suffices to show that $\widetilde{h}'$ is strictly positive on $B'$. We have
	\begin{align}
	\widetilde{h}'(y) = - \frac{u^3}{3N^2}y + \frac{1}{2} \frac{u^3}{N^2} +  \frac{u^2}{2N}.
	\end{align}
	Hence $\widetilde{h}'(y)>0$ is implied by the inequality
		\begin{equation}
	0 \leq y \leq \frac{3}{2} \frac{N}{u} + \frac{1}{2} 
	\end{equation}
	which is fulfilled for all $y \in B'$. Hence the drift $\EE{A_1|A_0=y}-y$ is negative for all states $y \in B'$ with minimal absolute value bounded from below by 
	\begin{equation}
	h(y_0) \geq 2 N^{1-2b-\eps'} (1+ O(N^{-\eps'})).
	\end{equation}
	which proves \eqref{DriftEst}.
\end{proof}
Similar as Lemma \ref{Lemma Coming down} we show now that the time to reach the  box $B$   from below is also no longer than polynomial with high probability.

\begin{lemma}[Going up from $1$] \label{Lemma going up from 1} \qquad \\
	Let $B$ and $\tau_B$ as in Lemma \ref{Lemma Coming down}. Then for $0<\eps < \eps' <\frac{1-2b}{3}$ and $\delta \leq \eps'-\eps.$ 
	\begin{align}
	\pp{ \tau_B \geq N^{2b+2\eps'}\mid A_0=1} = O((N^{1-2b+\eps})^{h_N}).
	\end{align}
\end{lemma}
\begin{proof}
	The strategy of the proof is similar to the proof of Lemma \ref{Lemma Coming down}: we estimate the drift in each point $y$ in the box $B'$ from below, with $B' := [1, \frac{2}{\rho^2 }  N^{1-b} -  N^{1-b-\eps' }]$. \\
	1. Assume we have shown for $y \in B'$ and $N$ large enough
	\begin{align}
	\EE{A_1|A_0=y} -y \geq y \frac{\rho^2}{4} N^{-b-\eps '} . \label{desired drift upwards}
	\end{align}
	Then
	\begin{align}
	\EE{A_{m_0}|A_0=1} \geq \inf \left\{ \left( 1+ \frac{\rho^2}{4} N^{-b-\eps '} \right)^{m_0} , \frac{2}{\rho^2 } N^{1-b} - N^{1-b-\eps '}\right\},
	\end{align}
	as long as $m_0$ is at most of polynomial order in $N$, since due to Proposition \ref{Lemma exp leaving time} after entering the box $B$ the process $A_m$ does not leave the box up to any polynomially long time with probability $1- O((N^{1-2b+\eps})^{h_N})$.
	In particular, 
	for $m_0 = \lceil  c N^{b+\eps ' }\ln N \rceil $ and some constant $c >0$, we have $\EE{A_{m_0}} \geq \frac{2}{\rho^2 } N^{1-b} - N^{1-b-\eps '}$.\\
	Now applying Markov's inequality yields
	\begin{align}
	\PP(A_{m_0} \leq \frac{2}{\rho^2 } N^{1-b} - N^{1-b-\eps }) & =	
	\PP(N- A_{m_0} \geq N- \frac{2}{\rho^2 } N^{1-b} + N^{1-b-\eps })\\
	&\leq \frac{\EE{N- A_{m_0}}}{N- \frac{2}{\rho^2 } N^{1-b} + N^{1-b-\eps}} 
	\leq \frac{ N - \frac{2}{\rho^2 } N^{1-b} + N^{1-b-\eps '}}{N- \frac{2}{\rho^2 } N^{1-b} + N^{1-b-\eps }} \\ &= 1- N^{-b-\eps}(1+ o(1)). \label{Iterate Markov}
	\end{align}
	In the worst case at time $m_0$ the process $A$ is still in state $1$. Iterating the argument in \eqref{Iterate Markov} yields that the probability that after time $\lceil N^{b+\eps'} \rceil m_0$ the process is still below $\frac{2}{\rho^2 } N^{1-b} - N^{1-b-\eps}$ is of order $O(\exp(-N^{\delta}))$, as claimed in the lemma.\\
	2. Now it remains to show \eqref{desired drift upwards}. For $y \in B'$ we observe that
		\begin{align}
	\EE{A_1|A_0=y}-y &= N (1 - \EE{(1-W_1)^H|A_0=y})-y \\
	&= N (1 -\EE{  \EE{(1-W_1)^H|A_0=y , W_1}|A_0=y})-y \\
	&=N \left(1 -\EE{\left( 1- \frac{W_1}{1-s_N +s_N W_1}\right)^y} \right)-y \\
	&= N \left(1 - \EE{ 1 - y \frac{W_1}{1-s_N +s_N W_1} + \binom{y}{2} \frac{W_1^2}{( 1-s_N +s_N W_1)^2} +O(y^3 W_1^3)  } \right) -y \\
	&\geq N \EE{ y \frac{W_1}{1-s_N +s_N W_1} - \binom{y}{2} \frac{W_1^2}{( 1-s_N )^2} +O(y^3 W_1^3)  }) -y . \label{Intermediate calculation result}
	\end{align}
	We analyse the first summand in the expectation separately, since the denominator and nominator both contain the random variable $W_1$.
	\begin{align}
	\EE{ y \frac{W_1}{1-s_N +s_N W_1}} &=  \EE{yW_1 (1+s_N +O(s_N W_1+ s_N^2) ) }\\
	&= \frac{y}{N} (1+s_N + O(s_N^2))
	\end{align}
	Using that $\EE{W_1^3}=O(N^{-3})$ (which follows from the assumptions of Theorem \ref{Haldane  Formel}), one can continue  \eqref{Intermediate calculation result} as
	\begin{align}
&= y (1+s_N +O(s_N^2)) -\binom{y}{2} \frac{\rho^2}{N} \frac{1}{( 1-s_N )^2} - y + O(y N^{-b-1}) +  O(y^3 \frac{1}{N^2})\\
	&\geq y \left(1+s_N +O(s_N^2) - \frac{ s_N - \frac{\rho^2 }{2} s_N N^{-\eps '} }{(1-s_N)^2} -1  + O(N^{-b-1}) + O(\frac{y^2}{N^2})\right)  \\
	&\geq  y\frac{\rho^2}{4} s_N N^{-\eps '},
	\end{align}
	where in the first inequality we used that $y \in B'$ is bounded from above. This gives the desired estimate \eqref{desired drift upwards} for the drift, and completes the proof of the lemma.
\end{proof}

\noindent
{\bf Completion of the proof of Theorem \ref{Haldane  Formel}b.}
\\
1. For proving \eqref{toshow} we will make use of Corollary \ref{corfix}, and to this purpose
	 derive asymptotic upper and lower bounds on the expectation of $A_{\rm eq}$ via stochastic comparison from above and below. Consider a time-stationary version $\mathcal A^{\rm stat}  = (A^{\rm  stat}_m)_{m\in \mathbb Z}$  and a  CASP $\mathcal A= (A_m)_{m \geq 0}$ that is started in $N$. We can couple both processes such that a.s. $A_m \geq A_m^{\rm stat}$  for all $m \geq 0$. This  implies
	\begin{align}
	\EE{A_{\rm eq}}=\EE{A_m^{\rm stat}} \leq \EE{A_m} ,\qquad m \geq 0.
	\end{align}
	Fix $0 < \alpha < \beta < \frac{2b-1}{3}$, and consider the box $B^{\alpha}=[\frac{2}{\rho^2} N^{1-b} \pm N^{1-b-\alpha}]$ as well as the (smaller) box $B^{\beta}=[\frac{2}{\rho^2}N^{1-b} \pm N^{1-b-\beta}]$.  Define $\tau_{B^{\beta}}:= \inf \{ m \geq 0 : A_m \in B^{\beta} \}$, the first hitting time of $B^{\beta}$, and $\tilde \tau_{B^{\alpha}}:= \inf \{m \geq \tau_{B^{\alpha}}: A_m \notin B^\alpha\}$, the first leaving time of $B^{\alpha}$.
	Choosing the time horizon $m_0= \lceil N^{2b+ 2\eps} \rceil$ with $0<\eps <1-b,$ we obtain
	\begin{align}
	\EE{A_{m_0}}&= \EE{A_{m_0}| \tilde \tau_{B^{\alpha}}> m_0 \geq \tau_{B^{\beta}} } \PP(\tilde \tau_{B^{\alpha}}> m_0 \geq \tau_{B^{\beta}}) \\
	&+\EE{A_{m_0}|m_0 < \tau_{B^{\beta}} } \PP(m_0 < \tau_{B^{\beta}}) \\
	&+\EE{A_{m_0}| m_0 > \tilde \tau_{B^{\alpha}} \geq \tau_{B^{\beta}} } \PP( m_0 > \tilde \tau_{B^{\alpha}} \geq \tau_{B^{\beta}}).
	\end{align}
	By Lemma \ref{Lemma Coming down} the second summand on the right hand side is of order $O(N\exp( - N^{\delta}))$ and by Lemma \ref{Lemma exp leaving time} the third summand is 
	of order $O(N(N^{1-2b+\alpha})^{h_N})$. Concerning the first summand we obtain $\PP(\tilde \tau_{B^{\alpha}}> m_0 \geq \tau_{B^{\beta}})= 1- O(N(N^{1-2b+\alpha})^{h_N})$. Observing that   $A_{m_0} = \frac{2}{\rho^2} N^{1-b} (1+o(1))$ whenever $A_{m_0}\in B^{\alpha}$ we thus conclude
	\begin{align}
\EE{A_{m_0}}&= \EE{A_{m_0}| \tilde \tau_{B^{\alpha}}> m_0 \geq \tau_{B^{\beta}} } \PP(\tilde \tau_{B^{\alpha}}> m_0 \geq \tau_{B^{\beta}}) + O(N(N^{1-2b+\alpha})^{h_N}) \\
&= \frac{2}{\rho^2} N^{1-b} (1+o(1)).
\end{align}
 This yields the desired upper bound on $\EE{A_{\rm eq}^{(N)}}$. The same argument applies for the lower bound, where we use a CASP started in the state $1$ and apply Lemma \ref{Lemma going up from 1} instead of Lemma~\ref{Lemma Coming down}.  
 \\

 \noindent
 {\color{black}{\bf Proof of Lemma \ref{Lemma implying Conditions} }\\
 a) In the symmetric Dirichlet case, $W_1^{(N)}$ is Beta$(\alpha_N, (N-1)\alpha_N)$-distributed, with $n$-th moment  $\frac {\Gamma(\alpha_N+n)}{\Gamma(\alpha_N) }\frac{\Gamma(\alpha_N N)}{\Gamma(\alpha_N N+n)}$. From this it is easy to see that \eqref{third} and \eqref{moments} are satisfied. In particular, the second moment is $\frac{\alpha_N+1}{N(N \alpha_N +1)}$, which by assumption obeys \eqref{second}.  \\
 b) For the Dirichlet-type weights, assume that $\mathscr W^{(N)}$ is of the form \eqref{Dirtype} with $Y$ obeying Conditions (i) and (ii) in Lemma \ref{Lemma implying Conditions}b). It is straightforward to  see that \eqref{second} and \eqref{third} are satisfied with $\rho^2:= \EE{Y^2}/(\EE{Y})^2$.  We now set out to show \eqref{moments}. To this end we first observe that
 \begin{equation} \label{estW}
 \EE{ \left(W_1^{(N)}\right)^n } \le \EE{ \left(\frac{Y_1}{Y_2+\cdots+Y_N}\right)^n} \le \EE{Y_1^n} \left(\frac {2}{N}
\right)^{n}
  \EE{\left(\frac {N-1}{Y_2+\cdots+ Y_N}\right)^n}.
 \end{equation}
 To bound the first factor  $\EE{Y_1^n}$ from above, we note, using Condition (i) of Lemma \ref{Lemma implying Conditions} on the exponential moments of $Y_1$ and the assumption $n \le 2h_N$ in \eqref{moments}, that
 \begin{equation}
\infty >\EE{ e^{cY_1} }\ge \EE{\frac{(cY_1)^n}{n!}}\ge \left(\frac{c}{2h_N}\right)^n \EE{Y_1^n}.
 \end{equation}
\noindent 
From this we obtain that 
\begin{equation}\label{bound}
\E[Y_1^n]\leq \E[e^{cY_1}]\Big(\frac{2h_N}{c}\Big)^n.
\end{equation}
On the other hand, the last factor in \eqref{estW} equals $\EE{\exp\left(-n\log\left(\frac  {Y_2+\cdots+ Y_N}{N-1}\right)\right)}$, which by Jensen's inequality is bounded from above by 
\[\EE{\exp\left(-\frac n{N-1}\left(\log Y_2 + \cdots + \log Y_N\right)\right)} =
	\left( \EE{\exp\left(-\frac{n}{N-1} \log Y_1 \right)}\right)^{N-1}.\]
By a Taylor expansion, and using Condition (ii) which implies that $\EE{\exp(-\frac{n}{N-1} \log Y_1)}$ is finite for $N$ big enough, we obtain
	\[\left(  1-\frac{n}{N-1} \EE{\log Y_1} + O(n^2 N^{-2} ) \right)^{N-1}  
	 \leq \left( e^{| \EE{\log Y_1}| (1+ O( \frac{n}{N} ))} \right)^n.
\]
Plugging this together with \eqref{bound} into \eqref{estW} yields 
$$
\E[\left(W_1^{(N)}\right)^n]\leq \E[e^{cY_1}]\Big(\frac{4e^{|\E[\log(Y_1)] | (1+ O( \frac{n}{N} ))}}{c} \frac{h_N}{N}\Big)^n.
$$
This implies \eqref{moments} with $K=\max\{1,\frac{5}{c} e^{| \EE{\log Y_1} | }\}$.
\qed 
}
\\

\noindent
\textbf{Acknowledgements}\\

We are grateful to two anonymous referees for stimulating and helpful comments, and thank G\"{o}tz Kersting for fruitful discussions and Jan Lukas Igelbrink for generating Figure 1. 

A.G.C. acknowledges CONACYT Ciencia Basica A1-S-14615 and travel support through the SPP 1590 and C.P. and A.W. acknowledge partial support through DFG project WA 967/4-2 in the SPP 1590.

\end{document}